\documentclass[11pt]{amsart}

\usepackage{amssymb,amsthm,amsfonts,latexsym,mathtools}

\usepackage{amsmath}
\usepackage{mathrsfs}
\usepackage{stmaryrd}
\usepackage[utf8]{inputenc}
\usepackage{color}
\usepackage{enumitem}
\usepackage{tabularx}
\usepackage{graphicx}
\usepackage{accents}
\usepackage{bbm}

\usepackage{pbox}

\newlength{\dhatheight}

\input{xy}
\xyoption{all}
\xyoption{poly}
\usepackage[all]{xy}
\allowdisplaybreaks

\theoremstyle{plain}
\newtheorem{theorem}{Theorem}[section]
\newtheorem*{theorem*}{Theorem}
\newtheorem{lemma}[theorem]{Lemma}
\newtheorem{corollary}[theorem]{Corollary}
\newtheorem*{corollary*}{Corollary}
\newtheorem{proposition}[theorem]{Proposition}

\theoremstyle{definition}

\newtheorem{definition}[theorem]{Definition}

\theoremstyle{remark}
\newtheorem{remark}[theorem]{Remark}

 \def\KK{{\mathbb{K}}}
 \def\kk{{\mathbbm{k}}}
 \def\QQ{\mathbb{Q}}

 \newcommand\ZZ{{\mathbb{Z}}}
 
 \def\C{{\mathcal C}}
 
 \def\E{{\mathcal E}}
 \def\F{{\mathcal F}}
 \def\G{{\mathcal G}}

 \DeclareMathOperator\Lk{Lk}

 \newcommand{\GF}[1]{\mathbb{F}_{#1}}

 \DeclareMathOperator\SL{SL}

 \DeclareMathOperator{\Sym}{Sym}

 \def\redF{\widehat{\mathcal{F}}}

 \DeclareMathOperator\Rad{Rad}

 \DeclareMathOperator\Id{Id}

 \newcommand{\tq}{\mathrel{{\ensuremath{\: : \: }}}}

 \DeclareMathOperator\Stab{Stab}
 
 \def\Sp{\mathcal{S}_p}

 \def\groupiso{\cong}

 \newcommand\gen[1]{\left\langle#1\right\rangle}

\renewcommand{\Sp}{\operatorname{Sp}}

\newcommand{\CT}[1]{\circ_{#1}}

\begin{document}

\title{On the frame complex of symplectic spaces}
   %\author{Kevin I. Piterman, Volkmar Welker}
   %\address{Departamento de Matem\'atica \\IMAS-CONICET\\
% FCEyN, Universidad de Buenos Aires. Buenos Aires, Argentina.}
%\email{piterman@mathematik.uni-marburg.de, welker@mathematik.uni-marburg.de}

%\author{Kevin I. Piterman and Volkmar Welker}
\author{Kevin I. Piterman}
\address{Philipps-Universit\"at Marburg \\
Fachbereich Mathematik und Informatik \\
35032 Marburg, Germany}
%\email{kpiterman@dm.uba.ar,  welker@mathematik.uni-marburg.de}
\email{piterman@mathematik.uni-marburg.de}

% \author{Volkmar Welker}
% \address{Philipps-Universit\"at Marburg \\
% Fachbereich Mathematik und Informatik \\
% 35032 Marburg, Germany}
%\email{welker@mathematik.uni-marburg.de}

\begin{abstract}
For a symplectic space $V$ of dimension $2n$ over $\GF{q}$, we compute the eigenvalues of its orthogonality graph.
This is the simple graph with vertices the $2$-dimensional non-degenerate subspaces of $V$ and edges between orthogonal vertices.
As a consequence of Garland's method, we obtain vanishing results on the homology groups of the frame complex of $V$, which is the clique complex of this graph.
We conclude that if $n < q+3$ then the poset of frames of size $\neq 0,n-1$, which is homotopy equivalent to the frame complex, is Cohen-Macaulay over a field of characteristic $0$.
However, we also show that this poset is not Cohen-Macaulay if the dimension is big enough.
\end{abstract}

\subjclass[2020]{05E45, 20J05, 51E24}

\keywords{Symplectic forms, frame complex, Cohen-Macaulay complex, graph spectrum}

\maketitle

\section{Introduction}

In \cite{PW22} we studied posets and complexes arising from finite-dimensional vector spaces equipped with non-degenerate $\epsilon$-Hermitian forms, with particular focus on Hermitian forms.
%Specifically, for a finite-dimensional vector space $V$ equipped with a non-degenerate Hermitian form over a field $\KK$, we described in \cite{PW22} the connectivity of its frame complex and studied its fundamental group.
%Recall that the frame complex of $V$ is the simplicial complex whose simplices are the sets of pairwise orthogonal $1$-dimensional non-degenerate subspaces of $V$.
More specifically, for a finite-dimensional vector space $V$ equipped with a non-degenerate Hermitian form over a field $\KK$, let $\F(V)$ be the frame complex of $V$.
This is the simplicial complex whose simplices are the sets of pairwise orthogonal $1$-dimensional non-degenerate subspaces of $V$.
We described in \cite{PW22} the connectivity of $\F(V)$ and studied its fundamental group.
For finite fields $\KK=\GF{q^2}$, we also computed the eigenvalues of the orthogonality graph, i.e., the $1$-skeleton of the frame complex.
This allowed us to apply Garland's method on the frame complex and establish vanishing results on the homology groups with coefficients in fields of characteristic $0$.
Moreover, it is shown in \cite{PW22} that some of these vanishing results can be propagated to the poset of proper non-degenerate subspaces, the poset of orthogonal decompositions and also to certain $p$-subgroup poset of the underlying unitary group, where $p$ is a prime dividing $q+1$.
In particular, if $\dim(V)<q+1$, all these posets are Cohen-Macaulay over a field of characteristic $0$.

In this article, we study the symplectic case.
Recall that a symplectic space is a vector space (of even dimension) equipped with a non-degenerate alternating bilinear form such that every vector is isotropic.
% Let $V$ be a vector space over a field $\KK$ equipped with a bilinear form $\Psi$.
% We say that $\Psi$ is alternating if $\Psi(v,w) = - \Psi(w,v)$ for all $v,w\in V$.
% In addition, if $\chara(F) = 2$ then we also require $\Psi(v,v) = 0$ for all $v\in V$.
% If $V$ is a finite-dimensional vector space over a field $\KK$ and $\Psi$ is a non-degenerate alternating form, we say that $(V,\Psi)$, or simply $V$, is a symplectic space.
% Recall that $V$ admits a symplectic space structure if and only if $\dim(V)$ is even.
% Moreover, the symplectic structure is unique up to isometry.

It was shown by K. Das \cite{Das} that the poset of proper non-degenerate subspaces of a symplectic space of dimension $2n$ over a finite field $\GF{q}$, $q\neq 2$, is (homotopy) Cohen-Macaulay of dimension $n-2$.
In particular, this implies that the poset of non-trivial orthogonal decompositions into non-degenerate subspaces is also Cohen-Macaualay of dimension $n-2$ (cf. \cite[Corollary 2.7]{PW22}).
However, little is known about the homotopy type of the associated frame complex.

Let $V$ be a symplectic space of dimension $2n$ over a field $\KK$.
The orthogonality graph of $V$ is the graph $\G(V)$ whose vertices are the $2$-dimensional non-degenerate subspaces of $V$, with edges given by the orthogonality relation.
The frame complex of $V$, denoted by $\F(V)$, is the clique complex of $\G(V)$.
Note that $\F(V)$ has dimension $n-1$.
Since every simplex of dimension $n-2$ is contained in a unique maximal simplex, $\F(V)$ collapses to an $(n-2)$-dimensional subcomplex.
To avoid non-canonical choices of the collapses, we sometimes work with the poset $\redF(V)$ of frames of size $\neq 0,n-1$, which has dimension $n-2$ and is homotopy equivalent to $\F(V)$ (see \cite[Lemma 2.10]{PW22}).
Therefore, we can analyse the Cohen-Macaulay property on the poset $\redF(V)$.

In this article, we focus on studying some properties of $\G(V)$ and $\F(V)$.
We summarise our main results in the following theorems.

\begin{theorem}
\label{main:thm}
    Let $V$ be a symplectic space of dimension $2n$ over a field $\KK$.
    \begin{enumerate}
        \item $\F(V)$ is connected if and only if $n\geq 3$.
        \item If $n\geq 5$ then $\F(V)$ is simply connected.
        \item If $\KK=\GF{q}$ is finite and $n < q+3$ then $\redF(V)$ is Cohen-Macaulay over a field of characteristic $0$.
        \item If $\KK=\GF{q}$ is finite and $n > q^2(q^2+1) + n(n-2)q^{-4}(q^4+q^2+1)^{-1}$ then $\redF(V)$ 
         is not Cohen-Macaulay over any field.
    \end{enumerate}
\end{theorem}

Indeed, item (3) of the previous theorem follows by an application of Garland's method.
To that end, we have computed first the eigenvalues of $\G(V)$.
Let $d_n = \frac{q^{2n-4}(q^{2n-2}-1)}{q^2-1}$, which is the degree of a vertex of $\G(V)$.

\begin{theorem}
\label{main:eigenvalues}
Let $V$ be a symplectic space of dimension $2n$ over $\GF{q}$, with $n\geq 2$.
The eigenvalues of $\G(V)$ are as follows:
\[\begin{cases}
    d_n, \, q^{2n-5}, \, \pm q^{2n-4}, \, \pm q^{3n-6} & n\geq 3,\\
    d_n, \, -q^{2n-4} & n = 2.
\end{cases}\]
\end{theorem}

We briefly describe the organisation of the paper.
Section 2 provides some preliminaries on symplectic spaces and the frame complex.
Sections 3 and 4 are more technical and contain the necessary background to compute the eigenvalues of $\G(V)$.
More concretely, in Section 3 we show that for two vertices $S,W$ of $\G(V)$, there are six possible configurations that characterise the relation between $S$, $W$ and the sum $S+W$ (see Definition \ref{def:sixGeneralCases}), and we give alternative algebraic descriptions of these possibilities in Table \ref{sixCases}.
In Section 4, for a fixed pair of vertices $S,W$ of $\G(V)$, we characterise the vertices $T$ adjacent to $S$ such that the behaviour of $T,W$ is as described in each of these six cases introduced in Section 3.
For $\KK$ finite, this allows us to explicitly compute the number of paths of a given length between $S$ and $W$, and see that these lengths only depend on which of these six cases hold for $S$ and $W$.
From this, in Section 5 we reduce the computation of the eigenvalues of the adjacency matrix of $\G(V)$ to compute the minimal polynomial of a certain $6$-dimensional matrix, from which we conclude Theorem \ref{main:eigenvalues} (see Theorem \ref{thm:eigenvalues}).

We prove items (1) and (2) of Theorem \ref{main:thm} in Sections 6 and 7 respectively.
In Section 8, we apply Garland's method to $\F(V)$ and obtain the conclusions of Theorem \ref{main:thm}(3).
Finally, we discuss some future directions and remarks in Section 9, where we establish item (4) of Theorem \ref{main:thm}.

\bigskip

\textbf{Acknowledgements.}
I would like to thank Volkmar Welker for all his support and the numerous discussions that we have had over the last year which led to the ideas presented here.
This article was written with the support of a fellowship granted by the Alexander von Humboldt Foundation, and some of the results were obtained during a research stay at Institut Mittag-Leffler in 2022.

\section{Symplectic spaces}

In this section, we recall some basic properties of symplectic spaces.
All vector spaces are finite-dimensional.
We refer the reader to Section 2 of \cite{PW22} for further details on definitions and notation.

From now on, by a \textit{symplectic space} we mean a vector space $V$ over a field $\KK$ equipped with a non-degenerate symplectic form $\Psi$.
Recall that the symplectic structure on $V$ is unique up to isometry (cf. \cite[(19.16)]{AscFGT}).
That is, there exists a basis $\{e_1,\ldots,e_{2n}\}$ of $V$ such that the matrix of the symplectic form $\Psi$ in this basis is given by
\begin{equation}
    \label{canonicalSymplecticMatrix}
    \mathbf{S}_n:= \left( \begin{matrix} 0 & \Id_n\\ -\Id_n & 0 \end{matrix}\right).
\end{equation}
Equivalently, for all $1\leq i,j\leq n$ we have
\[ \Psi(e_i,e_j) = 0 = \Psi(e_{n+i},e_{n+j}) \quad \text{ and } \quad \Psi(e_i,e_{n+j}) = \delta_{i,j} = -\Psi(e_{n+j},e_i).\]
% \[ \Psi(e_i,e_j) = 0 = \Psi(e_{n+i},e_{n+j}) \quad \forall 1\leq i,j\leq n,\]
% \[ \Psi(e_i,e_{n+j}) = \delta_{i,j} = -\Psi(e_{n+j},e_i) \quad \forall 1\leq i,j\leq n.\]
These kinds of bases are called \textit{symplectic bases}.
We will usually write $V$ for a symplectic space and the form $\Psi$ will be implicit.

More generally, we say that $(V,\Psi)$ is an \textit{$r$-degenerate symplectic space} if $\Psi$ is a bilinear alternating form with $\Psi(v,v) = 0$ for all $v\in V$ and such that $\dim(\Rad(V))=r$.

The following lemma summarises some well-known facts on symplectic spaces.
See for example \cite[Section 19]{AscFGT}.

\begin{lemma}
\label{lemmaBasicSymplecticSpaces}
Let $V$ be an $r$-degenerate symplectic space, and let $S,T\leq V$.
The following assertions hold:
\begin{enumerate}
    \item $V = \Rad(V)\oplus V'$ with $V'$ a symplectic space. %Moreover, $V/\Rad(V)$ has a natural symplectic structure making the canonical isomorphism $V'\to V/\Rad(V)$ an isometry.
    \item $S$ is an $s$-degenerate symplectic space, where $s = \dim(\Rad(S))$.
    \item $\dim(S) + \dim(S^\perp) = \dim(V) + \dim(S\cap \Rad(V))$.
    \item $(S^\perp)^\perp = S+\Rad(V)$.
    \item If $S$ is totally isotropic then $\dim(S) \leq (\dim(V)+\dim(S\cap \Rad(V)))/2$.
    \item $S$ is non-degenerate if and only if $S\oplus S^\perp = V$.
    \item (Witt's lemma) If $\alpha:S\to T$ is an isometry such that $\alpha(S\cap \Rad(V)) = T\cap \Rad(V)$, then $\alpha$ extends to an isometry of $V$.
\end{enumerate}
\end{lemma}

Denote by $\Sp(V)$ the group of isometries of a symplectic space $V$ of dimension $2n$ over $\KK$.
Since the symplectic structure is unique, this group depends only on $n$ and $\KK$, and so we may also write $\Sp_{2n}(\KK)$, or $\Sp_{2n}(q)$ if $\KK=\GF{q}$.
Recall that $|\Sp_{2n}(q)| = q^{n^2} \prod_{i=1}^n(q^{2i}-1)$.

% Below we recall some facts on the family of the symplectic groups $\Sp_{2n}(\KK)$.

% \begin{proposition}
% \label{propertiesSymplecticGroups}
% The following assertions hold:
% \begin{enumerate}
% \item If $M\in\GL_{2n}(\KK)$, then $ M \in \Sp_{2n}(\KK)$ if and only if $M^t \mathbf{S}_n M = \mathbf{S}_n$.
% \item $\Sp_{2n}(\KK)\leq \SL_{2n}(\KK)$, with equality if $n = 1$.
% \item $Z(\Sp_{2n}(\KK)) = \{\Id_{2n}, -\Id_{2n}\}$, so $|Z(\Sp_{2n}(\KK))| = \gcd(2,\chara(\KK)-1)$.
% \item $\PSp_{2n}(\KK) = \Sp_{2n}(\KK) / Z(\Sp_{2n}(\KK))$ is a simple group if and only if $(2n,\KK)\neq (2,\GF{2})$, $(2,\GF{3})$, $(4,\GF{2})$.
% \item $|\Sp_{2n}(q)| = q^{n^2} \prod_{i=1}^n(q^{2i}-1)$.
% \item If $q$ is even, then $\PSp_{2n}(q) \groupiso \Omega^+_{2n+1}(q)$.
% \end{enumerate}
% \end{proposition}

\medskip

Let $V$ be an $r$-degenerate symplectic space over $\KK$.
The \textit{orthogonality graph} $\G(V)$ is the graph whose vertices are the $2$-dimensional non-degenerate subspaces of $V$.
Two vertices of this graph are connected if and only if they are orthogonal.
We will usually identify $\G(V)$ with its set of vertices.
The \textit{frame complex} of $V$ is the clique complex of the orthogonality graph and we denote it by $\F(V)$.
A simplex $\sigma\in \F(V)$ of size $m$ is termed an \textit{$m$-frame}, or simply a \textit{frame}, and it consists of $m$ pairwise orthogonal $2$-dimensional non-degenerate subspaces of $V$.
We denote by $\F_m(V)$ the set of frames of size $m$.
Note that $\F(V)$ is a simplicial complex of dimension $n-1$.

By Witt's lemma, $\Sp(V)$ is transitive on the set of frames of the same size, and the stabiliser of the action can be easily computed.
That is, if $\sigma$ is an $m$-frame, then
\begin{equation}
\label{eq:stabilizer}
\Stab_{\Sp(V)}(\sigma) \groupiso \left( \SL_2(\KK)^m \rtimes \Sym_m\right) \times \Sp_{2(n-m)}(\KK).
\end{equation}
Here we have used that $\Sp_2(\KK) = \SL_2(\KK)$.
In particular, when $\KK = \GF{q}$ is a finite field, we obtain a nice expression for the reduced Euler characteristic of $\F(V)$.

% It is easy to verify, by Witt's lemma above, that $\Sp(V)$ is transitive on the sets of flags and frames of the same type.
% \begin{enumerate}
%     \item For two flags $\sigma_1 = (S_0 < S_1 < \ldots < S_m)$, $\sigma_2 = (T_0<T_1<\ldots<T_{m'})$ of non-degenerate subspaces, there exists $\alpha\in\Sp(V)$ such that $\alpha(\sigma_1) = \sigma_2$ (and hence $\alpha(S_i) = T_i$ for all $i$) if and only if $m=m'$ and $\dim(S_i) = \dim(T_i)$ for all $i$.
%     Thus, the orbits of flags just depend on their dimension vectors.
%     Therefore, if $e_{-1} = 0 \leq e_0 < e_1 < \ldots < e_m \leq 2n = e_{m+1}$ is a sequence of even dimensions, the stabilizer of a flag with dimension vector $(e_0,e_1,\ldots,e_m)$ is isomorphic to
%     \[ \prod_{i=0}^{m+1} \Sp_{e_i-e_{i-1}}(\KK).\]
%     \item An $m$-frame of $V$ is a set $\{S_1,\ldots,S_m\}$ of pairwise orthogonal and $2$-dimensional non-degenerate subspaces of $V$.
%     Then $\Sp(V)$ is transitive on the set $\F_m(V)$ of $m$-frames of $V$, and the stabilizer is isomorphic to
%     \[ \left( \Sp_{2}(\KK)^m \rtimes \Sym_m\right) \times \Sp_{2n-2m}(\KK).\]
% \end{enumerate}

\begin{corollary}
\label{eulerCharFramesSymplectic}
Let $V$ be a symplectic space of dimension $2n$ over $\GF{q}$.
Then
\[\tilde{\chi}(\F(V)) = \sum_{m=0}^n (-1)^{m+1} |\F_m(V)| = \sum_{m=0}^n (-1)^{m+1} \frac{|\Sp_{2n}(q)|}{ q^m (q^2-1)^m m! |\Sp_{2(n-m)}(q)| } .\]
\end{corollary}

As noted in \cite[Lemma 2.10]{PW22}, if $V$ is non-degenerate of dimension $2n$, $\F(V)$ collapses to an $(n-2)$-dimensional subcomplex.
In fact, if we let $\redF(V)$ to denote the poset of frames of size $\neq 0,n-1$, then $\F(V)$ is homotopy equivalent to (the order complex of) $\redF(V)$, where the latter has dimension $n-2$.
In particular, we see that $\widetilde{H}_{n-1}(\F(V),\ZZ) = 0$.

\section{Counting pairs $(S,W)$ of non-degenerate $2$-dimensional subspaces}

In this section, we focus on studying the algebraic relations between two non-degenerate $2$-dimensional subspaces $S,W\in \G(V)$.
From now on, $V$ will denote a symplectic space of dimension $2n$, $n\geq 2$, over a field $\KK$.

\begin{remark}
[Symplectic bases in dimension $2$]
\label{rk:symplecticBases}
Let $W\in \G(V)$, so $W = \gen{w,u}$ and $\Psi(w,u) \neq 0$ since $\Rad(W) = 0$.
After re-scaling $w$, we suppose that $\Psi(w,u) = 1$.
%Recall that $\Psi(u,w) = -1$ by the skew-symmetric hypothesis.
%From now on, we are going to work in odd characteristic.
Then, the set of ordered pairs $(w',u')$ generating $W$ with $\Psi(w',u') = 1$ is given by the linear combinations
\[ \begin{cases}
w' = aw + bu,\\
u' = cw + du,
\end{cases}\]
such that $1 = \Psi(w',u') = ad - bc$.
Hence, such pairs are in one-to-one correspondence with the group of isometries of $W$, that is, the special linear group $\SL_2(\KK)$.
\end{remark}

% Next, for simplicity, we are going to work in $V = \KK^{2n}$ with the canonical symplectic form $\Psi = \Psi_e$ whose underlying matrix $\mathbf{S}_n$ is given in (\ref{canonicalSymplecticMatrix}).
% Let $e_1,\ldots,e_{2n}$ denote the canonical basis of $V$.
% Then, for all $v,w\in V$ we have
% \begin{equation}
%     \label{canonicalSymplecticForm}
%     \Psi_e(v,w) = \sum_{i=1}^n v_i w_{n+i} - \sum_{j=1}^n v_{n+j}w_j.
% \end{equation}
% For $1\leq i\leq n$, define $E_i := \gen{e_i, e_{n+i}}$, and note that $\Psi(e_i,e_{n+i}) = 1$.

%Note that by transitivity, it is enough to understand the case $S = E_1$.

Before we continue, we introduce the following useful and non-standard notation.
Let $v\in V$ and $S,W\in \G(V)$,
with symplectic bases $x,y\in S$ and $w,u\in W$.
%$S = \gen{x,y}$ and $W = \gen{w,u}$ with $\Psi(x,y) = 1$ and $\Psi(w,u) = 1$.
We define:
\begin{equation}
    \label{generalNotation}
    \begin{cases}
    v_x := \Psi(v,y),\\
    v_y := -\Psi(v,x),\\
    v_S := v - v_xx - v_yy,\\
    w\CT{S} u := 1 - \Psi(w_S,u_S) = w_x u_y - w_y u_x.
    \end{cases}
\end{equation}
The vector $v_S$ is the orthogonal projection of $v$ onto $S^\perp$, and $v-v_S$ is the projection of $v$ onto $S$.
These two vectors do not depend on the choice of the basis $\{x,y\}$.
Thus, the value $w\CT{S} u$ does not depend on the basis either.
However, $v_x$ and $v_y$ do depend on the choice of the basis.

In the following definition, we describe six pairwise disjoint possible configurations of a pair of vertices $S,W$ of $\G(V)$.
Our goal will be to fix $S$ and characterise each case by giving suitable algebraic relations between $S$ and $W$.

\begin{definition}
\label{def:sixGeneralCases}
For a fixed $S\in \G(V)$ and $1\leq i \leq 6$, define $\E_i(S)$ to be the set of $W\in \G(V)$ satisfying the conditions of item ($i$) below:
\begin{enumerate}
    \item $\dim(S+W) = 2$, that is, $S = W$.
    \item $\dim(S+W) = 3$.
    \item $\dim(S+W) = 4$ and $S+W$ is degenerate.
    \item $S\perp W$.
    \item $S+W$ is non-degenerate, $S\not\perp W$, and $w\CT{S} u = 0$.
    \item $\dim(S+W)=4$, $S+W$ is non-degenerate, $S\not\perp W$, and $w\CT{S} u \neq 0$.
\end{enumerate}
Note that $\dim(S+W) = 4$ in cases (4) and (5), and $S\neq W$ in case (5).
\end{definition}

The following observations will be useful to characterise cases (1-6) above.

\begin{remark}
\label{rk:radicalContainedInTilde}
Let $S,W\in \G(V)$ with symplectic bases $x,y$ and $w,u$ respectively.
%Suppose that $S = \gen{x,y}$, $\Psi(x,y)=1$, and $W = \gen{w,u}$, $\Psi(w,u) = 1$.
Then
\[ S+W = \gen{x,y,w_S,u_S} = S \oplus \gen{w_S,u_S}, \text{ and } \gen{w_S,u_S}\perp S.\]
%From the latter relation we deduce that $\Rad(S+W) \leq \gen{w_S,u_S}$, with equality when $S+W$ is degenerate.
Hence, $S+W$ is degenerate if and only if $S+W = S \oplus \Rad(S+W)$ with $\Rad(S+W) = \gen{w_S,u_S}$.
\end{remark}

Below we describe in more detail the sets $\E_1(S),\ldots,\E_6(S)$ and compute their sizes when $\KK$ is finite.
For $\KK$ infinite, $|\E_i(S)|$ can be $0$, $1$ or infinite (here we set $0.|\KK| = 0$ and $|\KK|^0=1$).
This can be derived from our proofs after taking extra care on the set bijections.
We omit details for the infinite case and leave the verification for the reader.

Since $\E_1(S) = \{S\}$, we will focus on describing the sets $\E_i(S)$ with $i\geq 2$.
% By transitivity, it will be enough to consider the case $S = E_1$, which will simplify some notation.
% So, in the subsequent proofs, we relax the notation as in Eq. (\ref{generalNotationCanonicalChoice}) and assume that $S = E_1$.

\bigskip

\textbf{Case (2).} $W\in \E_2(S)$ if and only if $W = \gen{w,u}$ with $\Psi(w,u) = 1$, $w_S\neq 0$ and $u_S = 0$.
For $\KK = \GF{q}$ we have
\begin{equation}
    \label{eqDegSumDim3}
    b_n:=  |\E_2(S)| = (q^{2n-2}-1)(q+1).
\end{equation}

\begin{proof}
Suppose first that $\dim(S+W) = 3$ and pick a symplectic basis $w,u$ of $W$.
By Remark \ref{rk:radicalContainedInTilde}, $\gen{w_S,u_S}$ is $1$-dimensional.
Hence, by changing $(w,u)$ by $(-u,w)$ if necessary, we can suppose without loss of generality that $w_S\neq 0$ and $u_S = \lambda w_S$.
%so either $w_S\neq 0$ or $u_S\neq 0$.
In particular, we see that $\Psi(w_S,u_S) = 0$, leading to $1 = \Psi(w,u) = w\CT{S} u$.
Consider the basis $w' = w$, $u' = u - \lambda w$.
Then $\Psi(w',u') =  \Psi(w,u) = 1$, $w'_S= w_S\neq 0$ and $u'_S = u_S - \lambda w_S = 0$.
That is, $w',u'$ is a basis of $W$ as in the statement.

Reciprocally, if $W = \gen{w,u}$ with $\Psi(w,u) = 1$, $w_S\neq 0$ and $u_S = 0$, then $S+W = S\oplus\gen{w_S}$ has dimension $3$.

Now we count the number of such $W$.
To this aim, we count the number of pairs of vectors $(w,u)$ satisfying the conditions of the equivalence of this case, and then divide by the number of such pairs contained in a fixed $W$.
We fix a basis $S = \gen{x,y}$ with $\Psi(x,y)=1$.

First, the number of pairs $((w_x,w_y), (u_x,u_y))$ with $w_xu_y-w_yu_x = 1$ is exactly the number of bases in a $2$-dimensional symplectic space.
By Remark \ref{rk:symplecticBases} this number is just $|\SL_2(\KK)|$.

Second, $w = w_xx + w_yy + w_S$, where $w_S\in S^\perp\setminus 0$ is arbitrary, and $u = u_xx + u_yy$.
This gives $|\KK|^{2n-2}-1$ possibilities for $w_S$.
Therefore, the number of pairs $(w,u)$ with the properties $w\CT{S} u = \Psi(w,u) = 1$, $w_S\neq 0$ and $u_S = 0$ is $(|\KK|^{2n-2}-1)|\SL_2(\KK)|$.

Finally, suppose that $W = \gen{w,u}$, with $(w,u)$ as above, and $w' = aw+bu$, $u'=cw+du$ is also a pair satisfying the above conditions.
That is:
\[\begin{cases}
1 = \Psi(w',u') = ad - bc,\\
0\neq w'_S = a w_S, \Leftrightarrow a\neq 0,\\
0 = u'_S = c w_S, \Leftrightarrow c = 0.
\end{cases}\]
Hence, we see that $w' = aw + bu$, $u' = du$,  where $1 = ad$, i.e. $d = a^{-1}$.
Then we have $|\KK|-1$ possibilities for $a$, and $|\KK|$ for $b$.
This gives $|\KK|(|\KK|-1)$ pairs $(w',u')$ satisfying these conditions and inside $\gen{w,u}$.
This allows us to conclude that, for $\KK = \GF{q}$,
\[b_n = |\E_2(S)| = \frac{ (q^{2n-2}-1)q(q^2-1)}{q(q-1)} = (q^{2n-2}-1)(q+1).\]
\end{proof}

\textbf{Case (3).} $W\in \E_3(S)$ if and only if $W = \gen{w,u}$ with $\Psi(w,u)=1$, $w_S,u_S$ linearly independent and $\Psi(w_S,u_S) = 0$.
Moreover, for $\KK = \GF{q}$ we have
\[c_n := |\E_3(S)| = (q^{2n-2}-1) (q^{2n-3}-q).\]

\begin{proof}
Suppose that $W = \gen{w,u}$.
Then $S+W = S \oplus \gen{w_S,u_S}$ has dimension $4$ if and only if $w_S$, $u_S$ are linearly independent.
In addition, $S+W$ is degenerate if and only if $\gen{w_S,u_S}$ is totally isotropic, that is, $\Psi(w_S,u_S) = 0$ (see Remark \ref{rk:symplecticBases}).
Hence the equivalence of Case (3) holds.

Now we count the number of such $W$.
We proceed analogously as in the previous case.

Note that $1 = \Psi(w,u) = w\CT{S} u$.
This gives $|\SL_2(\KK)|$ possibilities for the pair of vectors $(w-w_S,u-u_S)$.
Next, we fix some non-zero $w_S\in S^\perp$, giving $|\KK|^{2n-2}-1$ possibilities.
Then $u_S$ is a non-zero vector in $\gen{w_S}^\perp \cap S^\perp = \gen{w_S,S}^\perp$, where the latter subspace has dimension $2n-3$.
However, since $w_S\perp w_S$, we need to exclude the scalar multiples of $w_S$ from the choices of $u_S$.
This gives $|\KK|(|\KK|^{2n-4}-1)$ possibilities for $u_S$.
In consequence, the number of pairs $(w,u)$ satisfying that $\Psi(w,u)=1$, $w_S,u_S$ are linearly independent and $\Psi(w_S,u_S) = 0$, is
\[  (|\KK|^{2n-2}-1)|\KK|(|\KK|^{2n-4}-1) |\SL_2(\KK)| .\]
Finally, to determine $|\E_3(W)|$, we count how many of these pairs generate the same subspace $W$.
Suppose that $(w,u)$ and $(w',u')$, with $w' = aw+bu$ and $u' = cw+du$, are two pairs that satisfies these conditions.
This means that
\[\begin{cases}
1 = \Psi(w',u') = ad-bc,\\
0 = \Psi(w'_S,u'_S) = \Psi(a w_S + bu_S,cw_S+du_S) & \text{(already holds),}\\
w'_S = aw_S + bu_S,\,\, u'_S= cw_S+du_S \text{ l.i., } \Leftrightarrow ad-bc\neq 0.
\end{cases}\]
Then the number of such pairs $(w',u')$ in a given $W$ is $|\SL_2(\KK)|$.
This shows that for $\KK = \GF{q}$, $c_n = (q^{2n-2}-1) (q^{2n-3}-q)$.
\end{proof}

\textbf{Case (4).} $\E_4(S) = \G(S^\perp)$, and for $\KK = \GF{q}$ we have
\[ d_n := |\E_4(S)| =  |\G(S^\perp)| = \frac{q^{2n-4} (q^{2n-2}-1)}{(q^2-1)}.\]

\begin{proof}
This is clear.
\end{proof}

Before moving to cases (5) and (6) of Definition \ref{def:sixGeneralCases}, we compute first $\E_5(S)\cup \E_6(S)$:

\bigskip

\textbf{Case (5)$\cup$(6).} $S+W$ is non-degenerate of dimension $4$ with $S\not\perp W$ if and only if $W = \gen{w,u}$ with $w\neq w_S$ and $\Psi(w_S,u_S) \neq 0$.
For $\KK = \GF{q}$ we have
\[ e_n^0 := |\E_5(S) \cup \E_6(S)| =  d_{n+1} - d_n - c_n - b_n - 1 = q^{2n-4}(q^{2n-2}-1)(q^2-q+1).\]

\begin{proof}
Note that $S\not\perp W$ if and only if, for $W = \gen{w,u}$ either $w\neq w_S$ or $u\neq u_S$.
Then, after changing the basis, we can suppose that $S\not \perp W$ if and only if $w\neq w_S$ for some basis $w,u$ of $W$.
Next, $\dim(S+W) = 4$ if and only if $w_S,u_S$ are linearly independent.
And, since $W\neq S$, $S+W$ is non-degenerate if and only if $\Psi(w_S,u_S)\neq 0$ (indeed this already implies that $w_S,u_S$ are linearly independent).
This gives the description of $W$ in the equivalence.

To count the number of such $W$, note that this is just the complement of the other four cases and that these five cases partition $\G(V)$, with $|\G(V)| = d_{n+1}$.
\end{proof}

Finally, to complete the characterisation of cases (5) and (6), we only need to characterise case (5) since case (6) can be regarded as the complementary case in $\E_5(S)\cup \E_6(S)$.

\bigskip

\textbf{Case (5).} $W\in \E_5(S)$ if and only if $W = \gen{w,u}$ with $w\neq w_S$, $u=u_S$ and $\Psi(w_S,u_S) = 1$.
The number of such $W$ for $\KK = \GF{q}$ is
\[ e_n^1 := |\E_5(S)| = q^{2n-4}(q^{2n-2}-1)(q+1).\]

\begin{proof}
The description is immediate from the union Case (5)$\cup$(6):
since the additional condition $w\CT{S} u = 0$ holds if and only if $u-u_S = \lambda (w-w_S)$, we can change $u$ by $u' = u - \lambda w$ and we get the desired basis.
Here we are using that $w\neq w_S$.

Next, we compute the size of $\E_5(S)$.
We have that $w-w_S$ is a non-zero vector of $S$, for which we have $|\KK|^2-1$ possibilities.
The vector $u=u_S$ is then completely determined from $u_S$.
On the other hand, the number of pairs $(w_S,u_S)$ with $\Psi(w_S,u_S) = 1$ is given by $|\SL_2(\KK)| |\G(S^\perp)|$.
Thus, we get $(|\KK|^2-1)|\SL_2(\KK)| |\G(S^\perp)|$ possible pairs $(w,u)$.

It remains to determine how many of them lie in the same fixed subspace $W = \gen{w,u}$.
Suppose that $w,u$ and $w',u'$ are two bases of $W$ satisfying the same requirements of this case.
Write $w'=aw+bu$ and $u'=cw+du$.
Then:
\[\begin{cases}
1 = \Psi(w',u') = ad-bc,\\
0\neq w'-w'_S = a(w-w_S),\\
0 = u'-u'_S = c(w-w_S).
\end{cases}\]
Thus $c = 0$, $d = a^{-1}$, and there are no restrictions on $b$.
This yields $|\KK|(|\KK|-1)$ possibilities for $(w',u')$.

In conclusion, for $\KK = \GF{q}$,
\[ e_n^1 = \frac{(q^2-1)|\SL_2(q)| |\G(S^\perp)|}{q(q-1)} = q^{2n-4}(q^{2n-2}-1)(q+1).\]
\end{proof}

For Case (6) we will not need a more concrete description since it will always fall in the complement of the cases (1-5).
From this, the size of $\E_6(S)$ can also be easily computed.

\bigskip

\textbf{Case (6).} For $\KK = \GF{q}$ we have
\[ e_n^2 := |\E_6(S)| = q^{2n-3}(q^{2n-2}-1)(q-2).\]

\begin{remark}
Note that the values $b_n,c_n,d_n,e_n^0,e_n^1,e_n^2$ are invariants of $n$ and $q$, and do not depend on the choice of $S$.
\end{remark}

We summarise the description of these five cases in Table \ref{sixCases}.

\setlength{\extrarowheight}{0.3cm}
\begin{table}[ht]
    \centering
    %\begin{tabular}{|@{ }p{3.7cm}@{ }|@{ }p{5.1cm}@{ }|@{ }p{5.3cm}@{ }|}
    \begin{tabular}{|c|c|c|}
    \hline
         $S\in \G(V)$ & {\pbox{4.8cm}{\centering Conditions on a basis $w,u$ of $W$ with $\Psi(w,u)=1$}} & Size of $\E_i(S)$ for $\KK = \GF{q}$\\[0.3cm]
         \hline
        $W\in \E_1(S) = \{S\}$ & $W = S$ & $1$ \\[0.3cm]
        \hline
        $W\in \E_2(S)$ & $w_S\neq 0 = u_S$ & $b_n = (q^{2n-2}-1)(q+1)$\\[0.3cm]
         \hline
        $W\in \E_3(S)$ & $w_S, u_S$ are li with $\Psi(w_S, u_S) = 0$ & $c_n = (q^{2n-2}-1)(q^{2n-3}-q)$\\[0.3cm]
         \hline
        $W\in \E_4(S)$ & $w=w_S$ and $u=u_S$ & $d_n = \frac{q^{2n-4}(q^{2n-2}-1)}{(q^2-1)}$ \\[0.3cm]
        \hline
        $W\in \E_5(S)\cup \E_6(S)$ & $\Psi(w_S,u_S) \neq 0$ and $w\neq w_S$ & $e_n^0 = q^{2n-4}(q^{2n-2}-1)(q^2-q+1)$\\[0.3cm]
        \hline
        $W\in \E_5(S)$  & $w\neq w_S$ and $u=u_S$ & $e_n^1 = q^{2n-4}(q^{2n-2}-1)(q+1)$\\[0.3cm]
        \hline
        $W\in \E_6(S)$  & $\Psi(w_S,u_S) \neq 0$ and $w\CT{S} u \neq 0$ & $e_n^2 = q^{2n-3}(q^{2n-2}-1)(q-2)$\\[0.3cm]
        \hline
    \end{tabular}
    \caption{Possible relations between two non-degenerate $2$-dimensional subspaces $S$ and $W$ of a $2n$-dimensional symplectic space $V$ over $\KK$, $n\geq 2$. The values in the third column are independent of the choice of $S$.}
    \label{sixCases}
\end{table}
\setlength{\extrarowheight}{0cm}

% \begin{remark}
% Note that if $w\CT{S} u = 0$ then $W\in \E_4(S)\cup \E_5(S)$: it is straightforward to check that $W\neq S$, and also $1 = \Psi(w,u) = \Psi(\tilde{w}^S,\tilde{u}^S)$. 
% Thus, by Table \ref{sixCases}, $W\notin \E_1(S) \cup \E_2(S)\cup \E_3(S) \cup \E_6(S)$, and this leads to $W\in \E_4(S)\cup \E_5(S)$.
% \end{remark}

Denote by $d_{n+1}$ the number of $2$-dimensional non-degenerate subspaces of a symplectic space of dimension $2n$ over $\KK$.
Note that this number depends only on $\KK$ and $n$.
In the following lemma, we compute the number of $2$-dimensional non-degenerate subspaces in an $r$-degenerate symplectic space in terms of the size of $\KK$ and $d_{n+1}$.

\begin{lemma}
\label{lemmaNr2dimNDSubspaces}
Let $V$ be an $r$-degenerate symplectic space of dimension $2n+r$.
Then the number of non-degenerate $2$-dimensional subspaces of $V$ is:
\[d_{n+1,r} := d_{n+1} \cdot |\KK|^{2r}.\]
\end{lemma}

\begin{proof}
Decompose $V = \Rad(V)\oplus V'$, where $V'$ is a symplectic space of dimension $2n$ by Lemma \ref{lemmaBasicSymplecticSpaces}.
Then the projection map $p:V\to V'$ has the property that, if $S\leq V'$ is non-degenerate, then $p^{-1}(S) = \Rad(V)\oplus S$.
Moreover, if $S\leq V'$ and $W\leq V$ are non-degenerate with $p(S) = p(W)$, then $W \leq \Rad(V)\oplus S$.
If in addition $W\leq V'$ then $W = S$ since $p|_{V'}$ is the identity map.
%On the other hand, if $S\leq V'$ is a $2$-dimensional non-degenerate subspace, then the number of $2$-dimensional non-degenerate subspaces of $\Rad(V)\oplus S$ depends only on the dimension of $\Rad(V)$.
Since $\G(\Rad(V)\oplus S)$ and $\G(\Rad(V)\oplus W)$ are isomorphic for any $S, W\in \G(V')$, we conclude that, for $S\in \G(V')$,
\[d_{n+1,r} = |\G(V)| = |\G(V')|  \cdot |\G(\Rad(V)\oplus S)|.\]

Let $T = R\oplus S$, where $R = \KK^{r}$ and $S = \KK^2$ is the canonical symplectic space.
Equip $T$ with an $r$-degenerate symplectic form $\Psi$ such that $R = \Rad(T)$ and $(S,\Psi|_S)$ has a symplectic basis $e_1,e_2$.
For vectors $v,v'\in R$ and constants $a,b,c,d\in \KK$, we see that $w = v +  (ae_1+be_2)$, $u = v' + (ce_1+d e_2)$ is a symplectic basis of a non-degenerate $2$-dimensional subspace $W = \gen{w,u}$ if and only if 
\[1 = \Psi(w,u) = \Psi(v + (ae_1+be_2), v' + (ce_1+d e_2) ) = \Psi(ae_1+be_2,ce_1+d e_2) =  ad-bc.\]
Hence, we have $|\SL_2(\KK)|  \cdot |\KK|^{2r}$ possible bases.
And, as we saw previously, in each subspace $W$ we have $|\SL_2(\KK)|$ such bases.
This gives the formula as in the statement of the lemma (note that the equality also holds in the infinite case).
\end{proof}

\section{Counting triples $(S\perp T,W)$ of non-degenerate $2$-dimensional subspaces}

Let $V$ be a symplectic space of dimension $2n$ over a field $\KK$.
In the previous section, we characterised and computed (when $\KK$ is finite) the number of subspaces $W\in \G(V)$ in each of the six cases described in Definition \ref{def:sixGeneralCases}.
See Table \ref{sixCases}.
Now we fix $S,W\in \G(V)$ and compute the number of $T\in \G(S^\perp)$ such that $W\in \E_j(T)$ for each value of $j\in \{1,\ldots,6\}$.

\begin{definition}
\label{tripleCasesNotation}
For $S,W\in \G(V)$ and $1\leq j\leq 6$, let
\[ \C_j(S,W) := \{ T\in \G(S^\perp)\tq W\in \E_j(T)\}, \text{ and} \]
%\vspace{-\baselineskip}
\[ \mu^n_j(S,W) := |\C_j(S,W)|.\]
\end{definition}

We will see that the values $\mu^n_j(S,W)$ depend just on the index $i\in \{1,\ldots,6\}$ for which $W\in \E_i(S)$.
%That is, they depend on the six cases of Definition \ref{def:sixGeneralCases}.
So for the rest of the section, we will focus on computing these numbers $\mu^n_j(S,W)$, which we abbreviate $\mu^n_j$ when the pair $(S,W)$ is implicit.
We will leave the technical details of the infinite case for the reader and mainly focus on the finite case $\KK = \GF{q}$ to obtain explicit formulas for the $\mu^n_j$.

%For simplicity, in the proofs given in this section, we will usually assume that $V$ is the canonical symplectic space of dimension $2n$ and $S = E_1$.
From now on, $V$ is a symplectic space of dimension $2n$ over $\KK$ and we fix $S,W\in \G(V)$.
In order to avoid trivial cases, we assume that $n\geq 2$.
The case $n = 1$ leads to $\G(V) = \{V\}$, so most of the sets we are computing are empty.

We proceed now with the computation of the values $\mu_j^n(S,W)$ for the six cases corresponding to the possible relations between $S$ and $W$ as in Definition \ref{def:sixGeneralCases}.

\begin{lemma}
[Case (1)]
Let $n\geq 2$ and $S\in \G(V)$.
Then, for $S = W$, we get:
\begin{equation}
    \label{muCase1}
    \mu_j^n(S,W) = \begin{cases}
    0   & j \neq 4,\\
    d_n & j = 4.
    \end{cases}
\end{equation}
\end{lemma}

\begin{proof}
This is straightforward from the definitions.
%since for every $T\in \F_1(S^\perp)$, $T+W = T+S$ is non-degenerate of dimension $4$ and $W\perp T$.
\end{proof}

Next, we characterise those $T\in \G(S^\perp)$ for which $W+T$ is degenerate.

\begin{proposition}
\label{propConditionsDegenerateTW}
Let $S, W\in\G(V)$ with $W\neq S$, and fix a symplectic basis $w,u$ of $W$.
Let $T\in \G(S^\perp)$.
Then $W+T$ is degenerate if and only if $T$ admits a basis $(x,y)$ such that the following conditions hold:
\begin{equation}
\label{conditionsDegenerateTW}
\begin{cases}
1 = \Psi(x,y),\\
\Psi(x,w) = 1 = \Psi(y,u),\\
\Psi(x,u) = 0 = \Psi(y,w).
\end{cases}
\end{equation}
Moreover, in that case, the ordered pair $(x,y)$ is unique.
\end{proposition}

\begin{proof}
We begin by proving the uniqueness.
Suppose that $x,y$ and $x',y'$ are two pairs satisfying the conditions of Eq. (\ref{conditionsDegenerateTW}).
Write $x' = ax+by$ and $y' = cx+dy$.
Then we have:
\[ \begin{cases}
 1 = \Psi(x',w) = a,\\
 1 = \Psi(y',u) = d,\\
 0 = \Psi(x',u) = b,\\
 0 = \Psi(y',w) = c.
\end{cases}\]
Hence $x'=x$ and $y'=y$.

Now we characterise when $W+T$ is degenerate in terms of $T$.
Take any symplectic basis $x,y$ of $T$.
Note that $x = x_S$ and $y = y_S$ since $T\perp S$.
Now, $W + T = \gen{w,u,x,y}$ is degenerate if and only if there exists $z = aw+bu+cx+dy \neq 0$ such that
\[\begin{cases}
0 = \Psi(z,w) = -b + c\Psi(x,w) + d\Psi(y,w),\\
0 = \Psi(z,u) = a + c\Psi(x,u) + d\Psi(y,u),\\
0 = \Psi(z,x) = -a\Psi(x,w) -b\Psi(x,u) - d,\\
0 = \Psi(z,y) = -a\Psi(y,w) -b\Psi(y,u) + c.
\end{cases}\]
Thus, this non-zero vector exists if and only if the determinant of the underlying matrix of the above system of equations is zero.
This determinant is:
\begin{equation*}
    \det\left(\begin{matrix} 0 & -1 & \Psi(x,w) & \Psi(y,w)\\
    1 & 0 & \Psi(x,u) & \Psi(y,u)\\
    -\Psi(x,w) & -\Psi(x,u) & 0 & -1\\
    -\Psi(y,w) & -\Psi(y,u) & 1 & 0
    \end{matrix}\right) = (1 - \Psi(x,w)\Psi(y,u) + \Psi(x,u)\Psi(y,w))^2.
\end{equation*}
Therefore, $W+T$ is degenerate if and only if
\begin{equation}
    \label{degenerateSumWT}
    1 = \Psi(x,w)\Psi(y,u) - \Psi(x,u)\Psi(y,w).
\end{equation}
%%%%% These commented lines contain the computation of the determinant
%\det\left( \begin{matrix} 1 & \Psi(x,u) & \Psi(y,u)\\
%    -\Psi(x,w) & 0 & -1\\
%    -\Psi(y,w) & 1 & 0
%    \end{matrix}\right) \\
%    \quad + \Psi(x,w) \det\left( \begin{matrix} 1 & 0 & \Psi(y,u)\\
%    -\Psi(x,w) & -\Psi(x,u) & -1\\
%    -\Psi(y,w) & -\Psi(y,u) & 0
%    \end{matrix}\right) - \Psi(y,w)\det\left( \begin{matrix} 1 & 0 & \Psi(x,u)\\
%    -\Psi(x,w) & -\Psi(x,u) & 0\\
%    -\Psi(y,w) & -\Psi(y,u) & 1
%    \end{matrix}\right)
% = \Psi(x,u)\Psi(y,w) + 1 - \Psi(x,w)\Psi(y,u) + \Psi(x,w)( - \Psi(y,u)\\
    %+ \Psi(x,w)\Psi(y,u)^2 - \Psi(x,u)\Psi(y,w)\Psi(y,u))\\
    %- \Psi(y,w)( - \Psi(x,u) + \Psi(x,u)\Psi(x,w)\Psi(y,u) - \Psi(x,u)^2\Psi(y,w) )\\
    %= 1 + 2\Psi(x,u)\Psi(y,w) - 2\Psi(x,w)\Psi(y,u) + \Psi(x,w)^2\Psi(y,u)^2 - 2\Psi(x,w)\Psi(x,u)\Psi(y,w)\Psi(y,u) + \Psi(x,u)^2\Psi(y,w)^2\\
%%%%%%%%%%%%%%%%%%%%%%%%%%%%%%%%%%%%
Define $x' := \Psi(y,u)x - \Psi(x,u)y$ and $y' := -\Psi(y,w)x + \Psi(x,w)y$.

Suppose that $W+T$ is degenerate.
Then Eq. (\ref{degenerateSumWT}) holds and
\[\begin{cases}
\Psi(x',y') = \Psi(x,w)\Psi(y,u) -\Psi(x,u)\Psi(y,w) = 1,\\
\Psi(x',w) = \Psi(y,u)\Psi(x,w) - \Psi(x,u)\Psi(y,w) = 1,\\
\Psi(x',u) = \Psi(y,u)\Psi(x,u) - \Psi(x,u)\Psi(y,u) = 0,\\
\Psi(y',w) = -\Psi(y,w)\Psi(x,w) + \Psi(x,w)\Psi(y,w) = 0,\\
\Psi(y',u) = -\Psi(y,w)\Psi(x,u) + \Psi(x,w)\Psi(y,u) = 1.
\end{cases}\]
Thus, the pair $(x',y')$ is a basis of $T$ and it satisfies the conditions of Eq. (\ref{conditionsDegenerateTW}).

Reciprocally, if $(x,y)$ is a pair satisfying the conditions of Eq. (\ref{conditionsDegenerateTW}), then Eq. (\ref{degenerateSumWT}) holds.
This finishes the proof of the proposition.
\end{proof}

The following corollary will be useful to count the number of subspaces $T$ for which $W+T$ is degenerate, in terms of counting elements in suitable orthogonal complements.

\begin{corollary}
\label{coroAlphaBeta}
Let $S,T,W$ be as in Proposition \ref{propConditionsDegenerateTW}.
Fix symplectic bases $e,f\in S$, $w,u\in W$ and $x,y\in T$.
Let $\alpha := -e + x$, $\beta := f+y$, $w' := f+w_S$ and $u' := e+u_S$.
Then $(x,y)$ satisfies the conditions of Eq. (\ref{conditionsDegenerateTW}) if and only if
\begin{equation}
    \label{alphaBetaRelations}
    \begin{cases}
    0 = \Psi(\alpha,\beta),\\
    0 = \Psi(\alpha,w'),\\
    0 = \Psi(\alpha,u'),\\
    0 = \Psi(\beta,w'),\\
    0 = \Psi(\beta,u').
    \end{cases}
\end{equation}
Reciprocally, if $\alpha,\beta\in V$ are such that $\alpha = -e+ \alpha_S$, $\beta = f+\beta_S$, and $\alpha,\beta$ verify the conditions in Eq. (\ref{alphaBetaRelations}), then $(x,y) = (\alpha_S,\beta_S)$ is a pair satisfying the conditions of Eq. (\ref{conditionsDegenerateTW}).

In particular, we see that $\alpha,\beta \in V$ satisfy Eq. (\ref{alphaBetaRelations}) with $\alpha = -e+\alpha_S$ and $\beta = f+\beta_S$, if and only if the following conditions hold:
\begin{equation}
    \label{alphaBetaOrthogonalComplements}
    \begin{cases}
    \alpha\in \gen{w',u',e}^\perp - \gen{f}^\perp, \quad \alpha_e = -1,\\
    \beta\in \gen{w',u',f,\alpha}^\perp - \gen{e}^\perp, \quad \beta_f = 1.
    \end{cases}
\end{equation}
\end{corollary}

Thus, in order to characterise the cases where the sum $W+T$ is degenerate, we can look for the solutions $\alpha,\beta$ in Eq. (\ref{alphaBetaOrthogonalComplements}) of Corollary \ref{coroAlphaBeta}.

Before we continue with the computation of the $\mu_j^n(S,W)$, we will need to characterise when $w\CT{T} u = 0$, in terms of the choices of $S,T,W$.

\begin{lemma}
\label{lem:circTCondition}
Let $S,T,W\in \G(V)$ and fix a symplectic basis $w,u$ of $W$.
Then the following are equivalent:
\begin{enumerate}
    \item $T\perp S$ and $w\CT{T}u= 0$,
    \item $T\perp S$ and $T\cap W^\perp \neq 0$,
    \item $T\perp S+W$, or $T$ contains a basis $x,y$ such that:
\begin{enumerate}
    \item $x\in S^\perp - W^\perp$,
    \item $y\in S^\perp \cap W^\perp  - \gen{x}^\perp$,
    \item $\Psi(x,y) = 1$.
\end{enumerate}
    \item $T\perp S$ and $W \in \E_4(T)\cup \E_5(T)$.
\end{enumerate}
Moreover, the number of pairs $(x,y)$ as in item (3)(a,b,c) is $|\KK|(|\KK|-1)$.
Thus, the total number of subspaces $T$ such that $T\perp S$, $T\not\perp W$ and $w\CT{T}u =0$ is equal to $\mu_5^n(S,W)$.
Hence, for a finite field $\KK = \GF{q}$ we have:
% \begin{equation}
%     \label{computationOfMu5Psi0}
%     \left(q^{2n-2} - q^{2n-\dim(S+W)}\right)\cdot \left( q^{2n-\dim(S+W)} - q^{2n-\dim(S+W)-1} \right) \frac{1}{q(q-1)^2},
% \end{equation}
% if $\Psi(\tilde{w},\tilde{u}) = 0$;
% and the value of $\mu_5^n(S,W)$ is
% \begin{equation}
%     \label{computationOfMu5PsiNot0}
%     \left(q^{2n-2} - q^{2n-\dim(S+W)}  - q^{2}+1 \right)\cdot q^{2n-\dim(S+W)-2} \frac{1}{(q-1)}
% \end{equation}
% if $\Psi(\tilde{w},\tilde{u}) \neq 0$.

\begin{equation}
%\label{computationOfMu5Psi0}
%\label{computationOfMu5PsiNot0}
\label{computationOfMu5}
\mu^n_5(S,W) = \frac{q^{2n-\dim(S+W)-2}}{(q-1)} \cdot \begin{cases}
\left(q^{2n-2} - q^{2n-\dim(S+W)}\right)    & \Psi(w_S,u_S) = 0,\\
\left(q^{2n-2} - q^{2n-\dim(S+W)}  - q^{2}+1 \right) & \Psi(w_S,u_S) \neq 0.
\end{cases}
\end{equation}
%For an infinite field $\KK$, $|\{ T\in \G(S^\perp) \tq T\not\perp W, w\CT{T}u=0\} |= \mu_5^n(S,W) = |\KK|$ is infinite as well.
\end{lemma}

\begin{proof}
Note that items (1) and (4) are equivalent by Table \ref{sixCases} (with $T$ in the role of $S$ there).
Further, if $T\perp W$ then $w\CT{T}u = 0$ and so items (1-4) hold.
Also $T\perp S$ if and only if $x,y\in S^\perp$ for any basis $x,y$ of $T$.
Thus, from now on, we assume that $T\perp S$ and $T\not\perp W$, and prove that items (1), (2) and (3) are equivalent.

Fix some basis $x,y$ of $T$ with $\Psi(x,y) = 1$.
Recall that $w = w_xx + w_yy + w_T$, and similarly $u = u_xx + u_yy + u_T$.
Since $T\not\perp W$, we have $w\neq w_T$ or $u\neq u_T$.
Then we have the following equivalences:
\begin{align*}
    w\CT{T}u = 0 & \Leftrightarrow w_xx+w_yy \text{ and } u_xx+u_yy \text{ are linearly dependent}\\
%    & \Leftrightarrow (\Psi(w,y),-\Psi(w,x)), (\Psi(u,y),-\Psi(u,x)) \text{ are l.d.}\\
    & \Leftrightarrow \Psi(w,y)\Psi(u,x) = \Psi(w,x)\Psi(u,y)\\
    & \Leftrightarrow 0 = \Psi(w, \Psi(u,x)y - \Psi(u,y)x) = \Psi(u, \Psi(w,y)x-\Psi(w,x)y)\\
    & \Leftrightarrow 0 = \Psi(w, u_xx+u_yy) = \Psi(u, w_xx+w_yy).
\end{align*}
Then, $w\CT{T} u = 0$ if and only if $w_xx+w_yy,u_xx+u_yy\in T\cap W^\perp$.
Since at least one of these vectors is non-zero, we conclude that $T\cap W^\perp \neq 0$ if and only if $w\CT{T}u=0$.
Thus, items (1) and (2) are equivalent.

Suppose now that item (2) holds.
We have that $\dim(T\cap W^\perp) = 1$ since $T\not\perp W$.
Then we can pick a pair of non-zero elements $(x,y)$ of $T$ such that $x\in T - W^\perp$ and $y\in T\cap W^\perp$.
After re-scaling, we suppose that $\Psi(x,y)=1$.
In particular, $y\notin \gen{x}^\perp$.
Then $(x,y)$ satisfies conditions (a,b,c) of item (3).
This shows that item (2) implies item (3).

Now we prove that item (3) implies item (1). That is, we show that under conditions (3)(a,b,c), we get $w\CT{T} u = 0$.
Since we are assuming $T\not\perp W$, there exists a basis $(x,y)$ of $T$ satisfying conditions (a,b,c).
Then $\Psi(w,y) = 0 = \Psi(u,y)$, from which we deduce that $w_x = 0 = u_x$, so $w\CT{T} u =0$.
This proves that item (3) implies item (1).
This finishes the proof of the equivalences.

% The In particular part is a straightforward consequence of the descriptions in items (a,b,c).
% We just show that Eq. (\ref{numberTWithCircCondition}) holds since the number of such bases $(x,y)$ inside a given $T$ is $q(q-1)$.
Next, we prove that the number of pairs $(x,y)$ inside $T$ satisfying (3)(a,b,c) is $|\KK|(|\KK|-1)$.
Suppose that $(x,y)$ and $(x',y')$ are two pairs of elements of $T$ satisfying (a,b,c).
Write $x' = ax+by$ and $y'=cx+dy$.
Then:
\begin{itemize}
    \item $\Psi(x',w) = a \Psi(x,w)$, $\Psi(x',u) = a \Psi(x,u)$. Then, by condition (a), $a\neq 0$.
    \item By condition (b), $0=\Psi(y',w) = c \Psi(x,w)$, $0 = \Psi(y',u) = c \Psi(x,u)$.
    Then, by condition (a), $c = 0$.
    \item By (c), $1 = \Psi(x',y') = ad-bc$.
\end{itemize}
We conclude that $c = 0$, $b\in\KK$ has no restrictions, $a\neq 0$, and $d = a^{-1}$.
This gives $|\KK|(|\KK|-1)$ possibilities for $(x',y')$.

Finally, we count the total number of pairs $(x,y)$ for subspaces $T$ as in (3)(a,b,c) and conclude the value of $\mu_5^n(S,W)$.
We leave the details of the infinite case for the reader and assume that $\KK = \GF{q}$ is finite.
We have
\[|S^\perp-W^\perp| = |S^\perp| - |(S+W)^\perp| = q^{2n-2} - q^{2n-\dim(S+W)}\]
possibilities for $x$.
Next, the number of possibilities for $y$ is
\[ \frac{|(S+W)^\perp| - | (S+W+\gen{x})^\perp|}{q-1} = \frac{q^{2n-\dim(S+W)} - q^{2n-\dim(S+W+\gen{x})}}{q-1},\]
where the denominator $q-1$ comes from the suitable re-scaling of $y$ needed for the condition $\Psi(x,y) = 1$.
Further, the number of possibilities for $y$ depends on whether $x\in S+W$ or not.
So we need to characterise when $x\in S+W$.
%For simplicity, we work with $S = E_1$, so $S+W = \gen{e_1,e_{n+1},\tilde{w},\tilde{u}}$.
Being $x\in S^\perp$, we have that $x\in S+W = S\oplus \gen{w_S,u_S}$ if and only if $x \in \gen{w_S,u_S}$.
But if $x = aw_S+bu_S$, it has to verify that $x\notin W^\perp$.
This means that either $0\neq \Psi(x,w) = b\Psi(u_S,w_S)$ or $0\neq \Psi(x,u) = a\Psi(w_S,u_S)$.
Then, this choice of $x$ is possible if and only if $\Psi(w_S,u_S)\neq 0$, and in that case we have $|\gen{w_S,u_S}|-1=q^2-1$ possibilities for $x$.
From this observation, it is straightforward to verify that the value of $\mu_5^n$ is as indicated in (\ref{computationOfMu5}).
Note that if $x\in S+W$, then $S^\perp \cap W^\perp - \gen{x}^\perp = \emptyset$, so there is no possible choice for $y$.
\end{proof}

The following lemma is also straightforward:

\begin{lemma}
\label{lm:mu4}
For any $S,W\in\G(V)$ we have that $\mu_4(S,W) = |\G( (S+W)^\perp)|$.
\end{lemma}

Now we have all the necessary tools to finish the computation of the values of $\mu_j^n(S,W)$.

\begin{proposition}
[Case (2)]
Suppose that $W\in \E_2(S)$.
Then, for $n\geq 2$, we have:
\begin{equation}
    \label{muCase2}
    \begin{cases}
    \mu_1^n(S,W) = 0 = \mu_2^n(S,W) = \mu_3^n(S,W) = \mu_6^n(S,W),\\
    \mu_4^n(S,W) = d_{n-1} \cdot |\KK|^2,\\
    \mu_5^n(S,W) = |\KK|^{4n-8}.
    \end{cases}
\end{equation}
\end{proposition}

\begin{proof}
By Table \ref{sixCases}, there exists a basis $w,u$ of $W$ such that $\Psi(w,u) = 1$ and $w_S\neq 0 = u_S$.
Let $T\in \G(S^\perp)$.
Then $u\in S\leq T^\perp$ and thus $u_T=u\neq 0$.
In particular, $w\CT{T}u=1-\Psi(w_T,u_T) = 1-\Psi(w,u)=0$.
By the equivalences in Lemma \ref{lem:circTCondition} we conclude that $W \in \E_4(T)\cup \E_5(T)$, so
\[ \mu_1^n(S,W) = 0 = \mu_2^n(S,W) = \mu_3^n(S,W) = \mu_6^n(S,W).\]

On the other hand, by Lemma \ref{lm:mu4}, $\mu_4^n(S,W) = |\G( (S+W)^\perp )|$.
Since $\dim(S+W)=3$, $(S+W)^\perp$ is a $1$-degenerate symplectic space of dimension $2n-3$.
Thus Lemma \ref{lemmaNr2dimNDSubspaces} gives the value of $\mu_4^n(S,W)$.

Now we compute $\mu_5^n(S,W)$.
Since we already have $w\CT{T}u=0$ for $T\in\G(S^\perp)$, if we additionally require $T\not\perp W$, then by Lemma \ref{lem:circTCondition} we conclude that $\mu_5^n(S,W) = |\KK|^{4n-8}$, which also contemplates the case of infinite $\KK$.
\end{proof}

\begin{proposition}
[{Case (3)}]
\label{propCase3STW}
Let $S\in \G(V)$ and $W\in \E_3(S)$.
For $n\geq 3$, we have:
\begin{equation}
    \label{muCase3}
    \begin{cases}
    \mu_1^n(S,W) = 0 = \mu_2^n(S,W),\\
    \mu_3^n(S,W) = |\KK|^{4n-9},\\
    \mu_4^n(S,W) = d_{n-2} \cdot |\KK|^4,\\
    \mu_5^n(S,W) = |\KK|^{4n-10}(|\KK|+1),\\
    \mu_6^n(S,W) = |\KK|^{4n-9} (|\KK|-2).
    \end{cases}
\end{equation}
If $n = 2$, this case cannot arise and the above values are not defined.
\end{proposition}

\begin{proof}
We suppose first that $n\geq 3$.
Clearly $\mu_1^n(S,W) = 0$ since $W\not\perp S$.

On the other hand, by Table \ref{sixCases}, there exists a basis $w,u$ of $W$ such that $\Psi(w,u) = 1$, $w_S,u_S$ are linearly independent, and $\Psi(w_S,u_S) = 0$.

We prove that $W+T$ always has dimension $4$ for $T\in\G(S^\perp)$.
Pick any basis $x,y$ of $T$, so $W+T = \gen{w,u,x,y}$.
Take a linear combination $0 = aw + bu+cx+dy$.
We show that $a=b=c=d=0$.
The projection onto $S$ gives $0 = a(w-w_S) + b(u-u_S)$.
Since $1 = \Psi(w,u) = \Psi(w_S,u_S) + w\CT{S} u = w\CT{S} u$,
%which is the determinant of the matrix with columns the coordinates of the vectors $w-w_S$ and $u-u_S$ written in a suitable basis of $S$.
$w-w_S$ and $u-u_S$ are linearly independent and so $a=0=b$.
Then also $c=d=0$ since $x,y$ are linearly independent.
Therefore $\dim(W+T) = 4$.
In particular, we have shown that $\mu_2^n(S,W) = 0$, and $\mu_3^n(S,W)$ is the number of $T\perp S$ such that $W+T$ is degenerate.

Next, in order to compute $\mu_3^n(S,W)$, we are going to count the number of solutions $\alpha,\beta$ of Eq. (\ref{alphaBetaOrthogonalComplements}) in Corollary \ref{coroAlphaBeta}.
Fix a basis $S = \gen{e,f}$ with $\Psi(e,f)=1$.
Recall that $w' = f+w_S$ and $u'=e+u_S$.
We assume that $\KK$ is finite and leave the details of the infinite case to the reader.

\bigskip
\textbf{Claim 1.} There are $|\KK|^{2n-4}$ solutions for $\alpha$.

\begin{proof}
We have that $\alpha\in \gen{w',u',e}^\perp - \gen{f}^\perp$ with $\alpha_e = -1$.
Thus the number of such $\alpha$ is equal to:
\[ \frac{1}{|\KK|-1}|\gen{w',u',e}^\perp - \gen{f}^\perp| = \frac{1}{|\KK|-1}\left( |\gen{w',u',e}^\perp| - |\gen{w',u',e,f}^\perp|\right).\]
We compute the dimensions of these subspaces.
First, $\gen{w',u',e} = \gen{f+w_S,u_S,e}$ has dimension $3$ since $u_S\neq 0$.
Thus, $|\gen{w',u',e}^\perp| = |\KK|^{2n-3}$.
Second, $\gen{w',u',e,f} = \gen{w_S,u_S,e,f}=S+W$, which has dimension $4$ by hypothesis.
Then $|\gen{w',u',e,f}^\perp| = |\KK|^{2n-4}$.
We conclude that the number of such $\alpha$ is $|\KK|^{2n-4}$.
\end{proof}

\bigskip
\textbf{Claim 2.} For each $\alpha$ satisfying Eq. (\ref{alphaBetaOrthogonalComplements}), there are $|\KK|^{2n-5}$ solutions for $\beta$.

\begin{proof}
We have that $\beta \in \gen{w',u',f,\alpha}^\perp - \gen{e}^\perp$ with $\beta_f = 1$.
Similar as in the previous case for $\alpha$, the number of solutions for $\beta$ is
\[ \frac{1}{|\KK|-1}|\gen{w',u',f,\alpha}^\perp - \gen{e}^\perp| = \frac{1}{|\KK|-1}\left( |\gen{w',u',f,\alpha}^\perp| - |\gen{w',u',e,f,\alpha}^\perp|\right).\]
Observe that $\gen{w',u',f,\alpha}$ has dimension $3$ if and only if ${\alpha} \in \gen{w',u',f}$.
However, this cannot hold since $\alpha \in \gen{w',u',e}^{\perp}$ and $\alpha_e\neq 0$.
Therefore, $|\gen{w',u',f,\alpha}^\perp|=|\KK|^{2n-4}$.

Next,
\[ \gen{w',u',e,f,\alpha} = \gen{w_S,u_S,e,f,\alpha_S} = \gen{w_S,u_S,\alpha_S}\oplus \gen{e,f}.\]
%Recall that $\tilde{w},\tilde{u}$ are linearly independent.
This subspace has dimension $5$ if and only if $\alpha_S\notin \gen{w_S,u_S}$, which must hold by definition of $\alpha$.
Therefore, $\gen{w',u',e,f,\alpha,}$ has dimension $5$ and
\[ |\gen{w',u',e,f,\alpha}^\perp| = |\KK|^{2n-5}.\]
This gives a total of $|\KK|^{2n-5}$ solutions for $\beta$ after fixing $\alpha$.
\end{proof}

By Claims 1 and 2, we conclude that the number of $T\perp S$ such that $W+T$ is degenerate is $|\KK|^{2n-4}\cdot |\KK|^{2n-5}=|\KK|^{4n-9}$.
In consequence, $\mu_3^n(S,W) = |\KK|^{4n-9}$.

Finally, since $\dim(S+W) = 4$ and $S+W$ is degenerate, $(S+W)^\perp$ is a $2$-degenerate symplectic space of dimension $2n-4=2(n-3)+2$.
This shows that
\[ \mu_4^n(S,W) = |\G( (S+W)^\perp)| = d_{n-2,2} = d_{n-2}\cdot |\KK|^4. \]
The value of $\mu_5^n(S,W)$ follows by Eq. (\ref{computationOfMu5}) in Lemma \ref{lem:circTCondition} since $\Psi(w_S,u_S)=0$.
The value of $\mu_6^n(S,W)$ follows by taking complements.

To finish the proof of this proposition, note that if $n=2$ and $\dim(S+W) = 4$, then $S+W = V$ which is non-degenerate.
Therefore, Case (3) does not arise in dimension $2n=4$ (cf. the value of $c_2$ in Table \ref{sixCases}).
\end{proof}

We state the remaining cases (4), (5) and (6) only for finite $\KK$.
Similar formulas can be obtained for infinite fields but they involve more elaborated arguments and we will not need them in this article.

\begin{proposition}
[{Case (4)}]
\label{propCase4STW}
Let $S\in \G(V)$ and $W\in \E_4(S) = \G(S^\perp)$.
For $n\geq 2$ and $\KK$ finite we have:
\begin{equation}
    \label{muCase4}
    \begin{cases}
    \mu_1^n(S,W) = 1,\\
    \mu_2^n(S,W) = b_{n-1} = (|\KK|^{2n-4}-1)(|\KK|+1),\\
    \mu_3^n(S,W) = c_{n-1} = |\KK|(|\KK|^{2n-4}-1)(|\KK|^{2n-6}-1),\\
    \mu_4^n(S,W) = d_{n-1},\\
    \mu_5^n(S,W) = e^1_{n-1} = |\KK|^{2n-6}(|\KK|^{2n-4}-1)(|\KK|+1),\\
    \mu_6^n(S,W) = e^2_{n-1} = |\KK|^{2n-5}(|\KK|^{2n-4}-1)(|\KK|-2).
\end{cases}
\end{equation}
\end{proposition}

\begin{proof}
Since $W\in \G(S^\perp)$, we can directly work in the orthogonal complement of $S$.
So the values of $\mu_j^n(S,W)$ correspond to the third column of Table \ref{sixCases} with $n-1$ in the role of $n$ there.
Note that these values are consistent in the case $n = 2$.
This concludes the proof of this case.
\end{proof}

\begin{proposition}
[{Cases (5) and (6)}]
\label{propCase5STW}
Let $S\in \G(V)$ and $W\in \E_5(S) \cup \E_6(S)$.
Then for $n\geq 2$ and $\KK$ finite we have:
\begin{equation}
    \label{muCases5and6}
    \begin{cases}
    \mu_1^n(S,W) = 0,\\
    \mu_2^n(S,W) = \begin{cases}
    |\KK|^{2n-4} & W\in \E_5(S),\\
    0 & W\in \E_6(S),\\
    \end{cases}\\
    \mu_3^n(S,W) = |\KK|^{2n-5}(|\KK|^{2n-4}-1),\\
    \mu_4^n(S,W) = d_{n-1},\\
    \mu_5^n(S,W) = |\KK|^{2n-6}(|\KK|^{2n-4}-1)(|\KK|+1),\\
    \mu_6^n(S,W) = \begin{cases}
    |\KK|^{2n-5}(|\KK|^{2n-4}-1)(|\KK|-2) & W\in \E_5(S),\\
    |\KK|^{2n-5}(|\KK|^{2n-4}(|\KK|-2)+2) & W\in \E_6(S).\\
    \end{cases}
    \end{cases}
\end{equation}
\end{proposition}

\begin{proof}
By Table \ref{sixCases}, we can fix a basis $w,u$ of $W$ such that
$\Psi(w_S,u_S)\neq 0$ and $w\neq w_S$.
In addition, if $w\CT{S} u = 0$, we take $u=u_S$.
Let $q$ be the size of $\KK$.
%, which is finite by hypothesis.

The value of $\mu_1^n$ is clear, and $\mu_4^n(S,W) = |\G((S+W)^\perp)| = d_{n-1}$ since $S+W$ is non-degenerate of dimension $4$.

We compute now $\mu_2^n$ and $\mu_3^n$.
By Corollary \ref{coroAlphaBeta}, it is enough to count the number of solutions $\alpha,\beta$ to Eq. (\ref{alphaBetaOrthogonalComplements}).
%This will give the total number of $T\in \G(S^\perp)$ such that $W+T$ is degenerate.
Let $S = \gen{e,f}$ with $\Psi(e,f) = 1$.
Recall that $w' = f + w_S$, $u' = e + u_S$.

\bigskip

\textbf{Claim 1.} There are $q^{2n-4}$ vectors $\alpha$ such that $\alpha\in \gen{w',u',e}^\perp - \gen{f}^\perp$ and $\alpha_e = -1$.

\begin{proof}
As in the proof of Claim 1 in Proposition \ref{propCase3STW}, we have to compute the dimension of the subspaces $\gen{w',u',e}$ and $\gen{w',u',e,f}$.

We have that $\gen{w',u',e} = \gen{f + w_S, u_S, e}$, and since $u_S\neq 0$, this subspace has dimension $3$ and  its orthogonal complement has dimension $2n-3$.

On the other hand, $\gen{w',u',e,f} = \gen{w_S,u_S}\oplus\gen{e,f}$.
Since $\Psi(w_S,u_S)\neq 0$ by hypothesis, this subspace has dimension $4$.
In consequence, there are $\frac{1}{q-1}\left( q^{2n-3} - q^{2n-4}\right) = q^{2n-4}$ choices for $\alpha$.
This finishes the proof of Claim 1.
\end{proof}

\textbf{Claim 2.} Let $\alpha\in \gen{w',u',e}^\perp - \gen{f}^\perp$ such that $\alpha_e = -1$.
Then, the number of vectors $\beta\in \gen{w',u',f,\alpha}^\perp - \gen{e}^\perp$ with $\beta_f = 1$ is given by:
\begin{enumerate}
    \item For $\alpha = -e + \Psi(u_S,w_S)^{-1}u_S$, we have $q^{2n-4}$ choices of $\beta$ if $w\CT{S} u = 0$ and $0$ otherwise.
    \item If $\alpha \neq -e + \Psi(u_S,w_S)^{-1}u_S$, we have $q^{2n-5}$ choices of $\beta$.
\end{enumerate}
\begin{proof}
Fix $\alpha$ satisfying the given conditions.
Recall that $\alpha = -e + \alpha_S$.
As before, we compute the dimensions of $\gen{w',u',f,\alpha}$ and $\gen{w',u',e,f,\alpha}$.

On one hand, we have $\gen{w',u',f,\alpha} = \gen{w_S,e + u_S,f,-e + \alpha_S}$.
This subspace has dimension $4$ if and only if $\alpha_S\neq -u_S$, that is, $\alpha\neq -u'$.
Also observe that $-u' \in \gen{w',u',e}^\perp - \gen{f}^\perp$ if and only if $w\CT{S} u = 0$.

Now we move to the second subspace:
\[ \gen{w',u',e,f,\alpha} = \gen{w_S,u_S,e,f,\alpha_S} = \gen{w_S,u_S,\alpha_S}\oplus \gen{e,f}.\]
This subspace has dimension $5$ if and only if $\alpha_S\notin \gen{w_S,u_S}$, since $w_S,u_S$ are linearly independent.
Otherwise, it has dimension $4$.

Suppose that $\alpha_S = a w_S+bu_S$.
The orthogonality relations of the definition of $\alpha$ yield:
\[ \begin{cases}
0 = \Psi(\alpha,w') = -1 - b \Psi(w_S,u_S),\\
0 = \Psi(\alpha,u') = a \Psi(w_S,u_S).
\end{cases}\]
From $\Psi(w_S,u_S)\neq 0$ we conclude that $a = 0$ and that $b = \Psi(u_S,w_S)^{-1}$.
Hence there exists a unique $\alpha$, given by $\alpha = -e + \Psi(u_S,w_S)^{-1}u_S$, such that $\gen{w',u',e,f,\alpha}$ has dimension $4$.
If $\alpha \neq -e + \Psi(u_S,w_S)^{-1}u_S$, then $\gen{w',u',e,f,\alpha}$ has dimension $5$.

To conclude the proof of this claim, we have to analyse the different cases we obtained.

Assume first that $\alpha = -e + \Psi(u_S,w_S)^{-1}u_S$ (which is indeed a possible solution for $\alpha$).
Then $\gen{w',u',e,f,\alpha}$ has dimension $4$.
However, $\gen{w',u',f,\alpha}$ has dimension $3$ if and only if $\Psi(u_S,w_S)^{-1} = -1$, if and only if $w\CT{S} u = 0$.
This shows that, with this fixed $\alpha$, if $w\CT{S} u = 0$ then we have $\frac{1}{q-1}\left( q^{2n-3} - q^{2n-4}\right) = q^{2n-4}$ choices for $\beta$.
If $w\CT{S} u\neq 0$, both subspaces have dimension $4$, and we get then $0$ choices for $\beta$.

Finally, if $\alpha\neq -e + \Psi(u_S,w_S)^{-1}u_S$, we obtain $q^{2n-5}$ choices for $\beta$.
\end{proof}

From Claims 1 and 2, it follows that the total number of subspaces $T\perp S$ such that $W+T$ is degenerate is given by $(q^{2n-4}-1)q^{2n-5} + h$, where $h = q^{2n-4}$ if $w\CT{S} u=0$ and $0$ otherwise.

To finish the proof of this proposition, we show that $W+T$ has dimension $3$ if and only if $w\CT{S} u=0$ and $x = -u_S$, where $x,y$ is a basis of $T$ as in Eq. (\ref{conditionsDegenerateTW}).
By Corollary \ref{coroAlphaBeta}, if $x=-u_S$ then $\alpha = -e - u_S$, so there are $q^{2n-4}$ choices for $\beta$ and hence for $y = \beta - f$.

Now, suppose that $W+T$ is degenerate, and take a 
basis $x,y$ of $T$ as in Eq. (\ref{conditionsDegenerateTW}).
Suppose that $0 = aw+bu+cx+dy$.
By projecting onto $S$ we get that $0 = a(w-w_S)+b(u-u_S)$.
Then, there is a nontrivial solution for $a,b$ if and only if $w\CT{S} u = 0$.
If $w\CT{S} u \neq 0$ then $a = 0 = b$, and $0 = cx+dy$ leads to $c=0=d$ too.
Hence, if $w\CT{S} u \neq 0$, i.e. $W\in \E_6(S)$, we get that $\dim(W+T) = 4$, and then $\mu_2^n(S,W) = 0$ and $\mu_3^n(S,W) = q^{2n-5}(q^{2n-4}-1)$.

Assume now that $w\CT{S} u = 0$, that is $W\in \E_5(S)$.
We show that $\mu_2^n(S,W) = q^{2n-4}$.
From the relations of Eq. (\ref{conditionsDegenerateTW}) we get:
\[\begin{cases}
0 = \Psi(aw+bu+cx+dy, w) = -b + c,\\
0 = \Psi(aw+bu+cx+dy, u) = a + d.
\end{cases}\]
Thus $c = b$, $d = -a$ and $0 = a(w - y) + b(u+x)$.
By projecting onto $S$ we get that $0 = a(w-w_S)$ since $u=u_S$ by hypothesis.
Also $w\neq w_S$ yields $a = 0$.
This shows that either $b = 0$ or $x = -u=-u_S$.
We conclude that $\dim(T+W) = 3$ if and only if $x = -u_S$, that is, $\alpha = -u'$.
In this case, by item (1) in Claim 2, we get $q^{2n-4}$ solutions for $\beta$, and hence for $y = \beta - f$.
Since the pair $(x,y)$ is unique with the properties of Eq. (\ref{conditionsDegenerateTW}), we have $\mu_2^n(S,W) = q^{2n-4}$.

The value of $\mu_5^n$ can be obtained from Eq. (\ref{computationOfMu5}) in Lemma \ref{lem:circTCondition} since $\Psi(w_S,u_S)\neq 0$.
The remaining value $\mu_6^n$ can be computed by taking suitable set complements.

Finally, note that these values are consistent in the case $n = 2$.
This finishes the proof of this proposition.
\end{proof}

Since we have seen that $\mu_j^n(S,W)$ depend just on the six cases in Definition \ref{def:sixGeneralCases}, at least for finite fields, we will abbreviate these values by $\mu_{i,j}^n$:

\begin{definition}
Let $V$ be a symplectic space of dimension $2n$ over $\GF{q}$.
For $1\leq i,j\leq 6$, define:
\[ \mu^n_{i,j}:= \begin{cases}
\mu^n_j(S,W) & \text{ there exist $S,W\in \G(V)$ with } W\in \E_i(S),\\
0 & \text{ otherwise}.
\end{cases}\]
\end{definition}

%Observe that, by transitivity, there exist $S,W\in \G(V)$ with $W\in\E_i(S)$ if and only if $|\E_i(E_1)| > 0$.

In the following theorem, we summarise the computations of this section in terms of the values $b_n,c_n,d_n,e^1_n,e^2_n$ computed in Table \ref{sixCases}.

\begin{theorem}
\label{thm:valuesMunji}
Let $V$ be a symplectic space of dimension $2n$, $n\geq 2$, over $\GF{q}$.
Then Table \ref{tab:munji} summarises the values of $\mu^n_{i,j}$ when there exist $S,W\in\G(V)$ with $W\in \E_i(S)$.

\begin{table}[ht]
    \centering
    \begin{tabular}{|c|c|c|c|c|c|c|}
        \hline
        $\mu^n_{i,j}$ & $j = 1$ & $j = 2$ & $j = 3$ & $j = 4$ & $j = 5$ & $j = 6$\\
        \hline
        $i = 1$ & $0$ & $0$ & $0$ & $d_n$ & $0$ & $0$\\
        $i = 2$ & $0$ & $0$ & $0$ & $d_{n-1}\cdot q^2$ & $q^{4n-8}$ & $0$ \\
        $i = 3$ & $0$ & $0$ & $q^{4n-9}$ & $d_{n-2}\cdot q^4$ & $q^{4n-10}(q+1)$ & $q^{4n-9}(q-2)$\\
        $i = 4$ & $1$ & $b_{n-1}$ & $c_{n-1}$ & $d_{n-1}$ & $e^1_{n-1}$ & $e^2_{n-1}$ \\
        $i = 5$ & $0$ & $q^{2n-4}$ & $q^{2n-5}(q^{2n-4}-1)$ & $d_{n-1}$ & $e^1_{n-1}$ & $e^2_{n-1}$\\
        $i = 6$ & $0$ & $0$ & $q^{2n-5}(q^{2n-4}-1)$ & $d_{n-1}$ & $e^1_{n-1}$ & $e^2_{n-1}+q^{2n-4}$\\
        \hline
    \end{tabular}
    \caption{The entry $(i,j)$ gives the value of $\mu_{i,j}^n$ in case there exist $S,W\in \G(V)$ with $W\in \E_i(S)$.}
    \label{tab:munji}
\end{table}

\end{theorem}

%%%%%%%%%%%%%%%%
%%%% Maple code
% bb:= (m,q) -> (q^(2*m-2)-1)*(q+1);
% cc:= (m,q) -> (q^(2*m-2)-1)*(q^(2*m-3)-q);
% dd:= (m,q)-> q^(2*m-4)*(q^(2*m-2)-1)/(q^2-1);
% ee1:= (m,q) -> dd(m,q)*(q^2-1)*(q+1);
% ee2:= (m,q) -> q*ee1(m,q)/(q+1)*(q-2);
% mu := (m,q)-> Matrix( [
% [0, 0, 0, dd(m,q), 0, 0],
% [0, 0, 0, dd(m-1,q)*q^2, q^(4*m-8), 0],
% [0, 0, q^(4*m-9), dd(m-2,q)*q^4, q^(4*m-10)*(q+1), q^(4*m-9)*(q-2)],
% [1, bb(m-1,q), cc(m-1,q), dd(m-1,q), ee1(m-1,q), ee2(m-1,q)],
% [0, q^(2*m-4), q^(2*m-5)*(q^(2*m-4)-1), dd(m-1,q), ee1(m-1,q), ee2(m-1,q)],
% [0, 0, q^(2*m-5)*(q^(2*m-4)-1), dd(m-1,q), ee1(m-1,q), q^(2*m-4) + ee2(m-1,q)]
% ] );
% p0:=(x,m,q)-> (x-dd(m,q)).(x-q^(2*m-5)).(x-q^(2*m-4)).(x+q^(2*m-4)).(x-q^(3*m-6)).(x+q^(3*m-6));
% p1:=(x,m,q)-> (x-dd(m,q)).(x-q^(2*m-4)).(x+q^(2*m-4)).(x-q^(3*m-6)).(x+q^(3*m-6));
% p2:=(x,m,q)-> (x-dd(m,q)).(x-q^(2*m-5)).(x+q^(2*m-4)).(x-q^(3*m-6)).(x+q^(3*m-6));
% p3:=(x,m,q)-> (x-dd(m,q)).(x-q^(2*m-5)).(x-q^(2*m-4)).(x-q^(3*m-6)).(x+q^(3*m-6));
% p4:=(x,m,q)-> (x-dd(m,q)).(x-q^(2*m-5)).(x-q^(2*m-4)).(x+q^(2*m-4)).(x+q^(3*m-6));
% p5:=(x,m,q)-> (x-dd(m,q)).(x-q^(2*m-5)).(x-q^(2*m-4)).(x+q^(2*m-4)).(x-q^(3*m-6));
% simplify(Column(p0(mu(m,q),m,q),1));
%%%%%%%%%%%%%%%%

\section{Walks and eigenvalues}

Throughout this section, $\KK = \GF{q}$ is a finite field and $V$ is a symplectic space over $\KK$.
For $S,W\in\G(V)$ and $r\geq 0$, let $l_r(S,W)$ denote the number of walks of length $r$ from $S$ to $W$ in the orthogonality graph.
Since our goal is to calculate the eigenvalues of $\G(V)$, in this section we prove that $l_r(S,W)$ depends just on the index $i$ for which $W\in \E_i(S)$, and then reduce the computation of eigenvalues to, roughly, the roots of the minimal polynomial of the $6$-dimensional matrix $\mu = (\mu^n_{i,j})_{i,j}$ (see Table \ref{tab:munji}).

First, we prove inductively that $l_r(S,W)$ depends just on the index $i$ for which $W\in \E_i(S)$.
Suppose that the values $l_k(S,W) = l_{i,k}^n$ are independent of the choices of $S,W\in\G(V)$ such that $W\in \E_i(S)$, for all $0\leq k < r$ and $1\leq i\leq 6$.
Then, for $1\leq i\leq 6$ and $W\in \E_i(S)$, we have:
\begin{equation}
    \label{pathLengthStrategy}
    l_r(S,W) = \sum_{T\perp S} l_{r-1}(T,W) = \sum_{j=1}^6 \sum_{T\in \C_j(S,W)} l_{r-1}(T,W) = \sum_{j=1}^6 \mu^n_{i,j} \cdot l_{j,r-1}^n.
\end{equation}
Thus $l_r(S,W)$ depends just on the index $i$ and not on the specific choice of $S,W$.

For example, we have that:
\begin{equation}
    \label{length1}
    l_1(S,W) = \begin{cases}
    1 & W \in \G(S^\perp) = \E_4(S),\\
    0 & W \notin \E_4(S).
    \end{cases}
\end{equation}
and 
\begin{equation}
    \label{length2}
    l_2(S,W) = \mu_4^n(S,W) = \begin{cases}
    d_n & W\in \E_1(S),\\
    d_{n-1} \cdot |\KK|^2 & W\in \E_2(S),\\
    d_{n-2} \cdot |\KK|^4 & W\in \E_3(S),\\
    d_{n-1} & W\in \E_4(S)\cup \E_5(S) \cup \E_6(S).
    \end{cases}
\end{equation}

\begin{definition}
Let $V$ be a symplectic space of dimension $2n$ over $\GF{q}$.
For $r\geq 0$ we define
\begin{align*}
    \mu & := (\mu_{i,j}^n)_{i,j},\\
    l^n_{i,r} & := \begin{cases}
        l_r(S,W) & \text{there exist }S,W\in \G(V)\text{ with } W\in \E_i(S),\\
        0 & \text{otherwise},
    \end{cases}\\
    l^n_r & := (l^n_{1,r},\ldots, l^n_{6,r})^t.
\end{align*}
Thus, for the adjacency matrix $A$ of $\G(V)$ and $W\in \E_i(S)$ we have:
\begin{equation}
\label{generalMuFormula}
    \begin{cases}
    l^n_{r} = \mu \cdot l^n_{r-1}= \mu^{r} \cdot l_0, & \forall r\geq 1,\\
    (A^r)_{(S,W)} = l^n_{i,r} = (\mu^r \cdot l_0)^t_i.
    \end{cases}
\end{equation}
\end{definition}

From these definitions and the already mentioned properties, we see that a polynomial $R$ annihilates the adjacency matrix $A$ of $\G(V)$ if and only if $R(\mu)\cdot l_0 = 0$.
In particular, the minimal polynomial of $A$ divides the minimal polynomial of $\mu$ and thus has degree at most $6$.

Using these observations, computations in MAPLE show that the minimal polynomial of $A$ has the following $6$ roots for $n\geq 3$:
\begin{equation}
\label{listEigenvalues}
%\begin{cases}
    d_n, q^{2n-5},\pm q^{2n-4}, \pm q^{3n-6}.
    %d_n, q^{2n-5},- q^{2n-4}, \pm q^{3n-6} & q = 2.
%\end{cases}
\end{equation}

In the case $n = 2$, we have that $\G(V)$ is a disconnected graph with as many connected components as $2$-frames.
Note that $d_2 = 1$ and the values of $l_r$ are:
\[l_0 = (1,0,0,0,0,0), \, l_1 = (0,0,0,1,0,0), \, l_2 = (1,0,0,0,0,0) = l_0.\]
Thus, $R = X^2-1$ is the minimal polynomial for the adjacency matrix of $\G(V)$.

We have proved:

\begin{theorem}
\label{thm:eigenvalues}
Let $V$ be a symplectic space of dimension $2n$ over $\GF{q}$, with $n\geq 2$.
The roots of the minimal polynomial of the adjacency matrix of $\G(V)$ are as follows:
\[\begin{cases}
    d_n, \, q^{2n-5}, \, \pm q^{2n-4}, \, \pm q^{3n-6} & n\geq 3,\\
    %d_n, \, q^{2n-5}, \, -q^{2n-4}, \, \pm q^{3n-6} & q = 2 \text{ and } n\geq 3,\\
    d_n, \, -q^{2n-4} & n = 2.
\end{cases}\]
\end{theorem}

%n=2
%Also, the multiplicities are straightforward to compute: since the multiplicity of $d_2=1$ is the number of connected components, which is $d_3/2$, and the multiplicity of $-1$ is $d_3-d_3/2=d_3/2$.

% \begin{align*}
%     c_0 & = - d_n q^{5n-4} (-1)^n ,\\
%     c_1 & = (q^{n-2}(-1)^n - d_n) q^{4n-2},\\
%     c_2 & = d_n(q^{2n}+q^{2n-2})q^{n-2}(-1)^n + q^{4n-2},\\
%     c_3 & =  (q^{2n}+q^{2n-2})(d_n-q^{n-2}(-1)^n),\\
%     c_4 & = -d_nq^{n-2}(-1)^n-(q^{2n}+q^{2n-2}),\\
%     c_5 & = -d_n+q^{n-2}(-1)^n,\\
%     c_6 & = 1.
% \end{align*}

\section{Connectivity}

In this section, we characterise the connectivity of the orthogonality graph $\G(V)$ and hence of the frame complex $\F(V)$ for a symplectic space $V$ of dimension $2n$ over a field $\KK$.

%Note that if $n = 1$ then $\F(V)$ contains just a vertex $V$ and thus $\F(V) = \G(V)$ is contractible.

% Let $V$ be a symplectic space of dimension $2n$ over $\GF{q}$.
% Then, if $n = 1$, $\F(V) = \{V\}$, and if $n = 2$ then $\F(V)$ is a disjoint union of $\frac{d_3}{2} = q^2(q^2+1)/2$ contractible connected components.
% Hence, from now on we assume that $n\geq 3$ in order to  study homotopy properties of $\F(V)$ when $n\geq 3$.

% Note that $\G(V)$ has diameter $2$ if and only if for all $S,W\in \G(V)$ we have that $\mu_4^n(S,W) = |\F( (S+W)^\perp )|\neq \emptyset$.
% Now, these are the values $\mu_{i,4}$, corresponding to the column $j=4$ of Table \ref{tab:munji}.
% We check when these are positive.
% Clearly $d_n > 0$.
% On the other hand, since $d_1 = 0$, we have that $d_{n-2} > 0$ if and only if $n\geq 4$.
% And similarly $d_{n-1} > 0$ if and only if $n\geq 3$.
% Thus $\G(V)$ is connected of diameter $2$ if and only if $n\geq 4$.

% So we analyse now the case $n = 3$.
% Here we have that $\mu_4^3(S,W) = 0$ when $W\in \E_3(S)$, that is, $S+W$ is degenerate of dimension $4$.
% That is, $\mu_{3,4}^3 = 0$.
% However, we will see that $l_3(S,W) > 0$ for any $S,W$.
% To compute the walks of length $3$, we use Eq. (\ref{generalMuFormula}):
% \begin{align*}
%     l_3 & = \mu^3 \cdot l_0\\
%     & = ( q^4+q^2 , q^4+q^2 , q^4-q^3+q^2 , q^5+3 q^4-2 q^3+q , 2 q^4-q^3+q , q^4-q^3+q^2+q)^t.
% \end{align*}
% Now it is straightforward to verify that each entry is positive for $q\geq 2$.

We have the following theorem:

\begin{theorem}
[{Connectivity}]
\label{thm:connected}
Let $V$ be a symplectic space of dimension $n\geq 2$ over $\KK$.
The following assertions hold.
\begin{enumerate}
    \item $\F(V)$ is connected if and only if $n\geq 3$.
    \item If $n = 2$, $\F(V)$ is homotopy-discrete with $|\KK|^2(|\KK|^2+1)/2$ connected components.
    \item If $n = 3$, $\G(V)$ has diameter $3$ and $\F(V)$ is homotopy equivalent to a wedge of $1$-spheres.
    In particular, for $\KK = \GF{q}$ finite, the number of such spheres is $-\tilde{\chi}(\F(V)) = \frac{1}{3}\left( q^{12}+2 q^{10} - q^8 - 2 q^6 - 3q^4+3\right)$.
    \item If $n\geq 4$, $\G(V)$ has diameter $2$.
\end{enumerate}
Moreover, if $V'$ is an $r$-degenerate symplectic space of dimension $2n+r$, then $\F(V')$ is connected if and only if $n\geq 3$.
\end{theorem}

\begin{proof}
For $n\geq 2$, $\G(V) = \F(V)$ and it is disconnected with a (contractible) connected component for each $2$-frame of $V$.
The number of $2$-frames is thus the number of vertices over $2$, which is $d_{3}/2$.
This number is infinite if $\KK$ is infinite, or, if $\KK = \GF{q}$, $d_3/2 = q^2(q^2+1)/2$ by Table \ref{sixCases}.

Now assume that $n\geq 4$.
The case $n = 3$ will be treated separately.
Let $S,W$ be two vertices of $\G(V)$.
Write $(S+W)^\perp = \Rad(S+W) \oplus U$, with $U$ symplectic of dimension $2m$ and $\Rad(S+W)$ of dimension $r$.
Then $2m = 2n-\dim(S+W)-r \geq 8-4-2=2$, which implies $\G(U)\neq\emptyset$.
Therefore
\[ l_2(S,W) = |\G( (S+W)^\perp )| \geq |\G(U)| > 0.\]
This shows that $\G(V)$ is connected of diameter exactly $2$ since there always exist non-adjacent vertices.

Finally, we deal with the case $n = 3$.
Take $S,W\in \G(V)$.
As before, $l_2(S,W) > 0$ if $(S+W)^\perp \neq \Rad(S+W)$.
However, since now the dimension is low, the equality $(S+W)^\perp = \Rad(S+W)$ holds if and only if $W\in \E_3(S)$.
In fact, $\E_3(S)\neq\emptyset$ for example by taking $S = \gen{e_1,e_4}$ and $W = \gen{e_1+e_2,e_4+e_2}$ in the canonical form.
Nevertheless, we show that $l_3(S,W) > 0$ for $W\in \E_3(S)$.
Indeed, $l_3(S,W) \geq l_2(T,W) > 0$ for $T\in \C_5(S,W)$, where the latter is non-empty by Proposition \ref{propCase3STW}.
%it is enough to show that $\C_j(S,W)$ is non-empty for some $j\neq 3$.
%By Proposition \ref{propCase3STW}, $|\C_5(S,W)|=\mu_5^n(S,W)>0$.
Therefore, $\G(V)$ is connected of diameter $3$.
Since $\F(V)$ is homotopy equivalent to a subcomplex of dimension $n-2=1$, we conclude that $\F(V)$ is homotopy equivalent to a wedge of $1$-spheres.
We can compute the number of such spheres for finite $\KK = \GF{q}$ from the reduced Euler characteristic of $\F(V)$, which is
\[\tilde{\chi}(\F(V)) =  -\frac{1}{3}\left( q^{12}+2 q^{10} - q^8 - 2 q^6 - 3q^4+3\right).\]

We have established items (1-4). It remains to prove the Moreover part.

Let $V'$ be an $r$-degenerate symplectic space of dimension $2n+r$.
Consider a complement $V$ to the radical, that is $V' = \Rad(V')\oplus V$, so that $V$ is a symplectic space of dimension $2n$.
Assume that $n\geq 3$, so $\F(V)$ is connected by item (1).
Let $S\in\G(V')$ and write $S = \gen{x,y}$ with $x = r + w$, $y = r'+u$ and $w,u\in V$.
Note that $0\neq \Psi(x,y) = \Psi(w,u)$.
Therefore $W = \gen{w,u} \in \G(V)$.
Now, consider some $T\in\G(V\cap W^\perp)$, which exists since $V$ is non-degenerate of dimension $2n\geq 6$.
Then it is easy to check that also $T\perp S$, that is $T\perp (S+W)$ and $T\in \G(V)$.
This shows that $\G(V')$, and hence $\F(V')$, are connected.

On the other hand, if $\F(V')$ is connected, then the projection $\G(V')\to \G(V)$ is surjective, so $\G(V)$ is connected.
Thus $n\geq 3$ by item (1).
\end{proof}

\section{Simple connectivity}

In this section, we show that $\F(V)$ is simply connected if $n\geq 5$, where $V$ is a symplectic space of dimension $2n$.
Recall that if $n = 3$ then $\F(V)$ is homotopy equivalent to a wedge of $1$-spheres by Theorem \ref{thm:connected}, so its fundamental group is a non-trivial free group.
However, the study of the case $n = 4$ remains open.
%Hence, it remains to understand the fundamental group of $\F(V)$ for the case $n = 4$.

In order to prove that $\pi_1(\F(V)) = 1$ if $n\geq 5$, we show that the minimal cycles of the graph $\G(V)$ are null-homotopic.
Concretely, recall that every loop in $\F(V)$ can be homotopically carried to a closed walk in the orthogonality graph $\G(V)$.
Since $\G(V)$ has diameter $2$ by Theorem \ref{thm:connected}, a closed walk in $\G(V)$ can be (homotopically) decomposed into cycles of length $r \leq 2\cdot 2 + 1 = 5$, called \textit{$r$-gons}.
Therefore, $\F(V)$ is simply connected if and only if every $r$-gon is null-homotopic.
Moreover, since $\F(V)$ is a clique complex, $3$-gons are filled and so null-homotopic.
Thus, it remains to show that $4$-gons and $5$-gons are also null-homotopic.
For more details, see Section 4 of \cite{PW22} or in \cite{ASegev}.

We show first that $5$-gons are homotopic to concatenations of squares (and thus possible $4$-gons).

\begin{lemma}
\label{lm:reduction5gonsDim5orMore}
Let $C$ be a $5$-gon in the graph $\G(V)$ of a symplectic space $V$ of dimension $2n$, with $n\geq 5$.
Then $C$ is homotopic in $\F(V)$ to a concatenation of squares.
\end{lemma}

\begin{proof}
Suppose that $C$ is the sequence of vertices $S_0,S_1,S_2,S_3,S_4,S_5=S_0$ where $S_i\perp S_{i+1}$ for all $0\leq i < 5$.
Then $S_0$ is not orthogonal to $S_2$ nor to $S_3$.
Let $T = S_0+S_2+S_3$.
Clearly $\dim(T)\leq 6$, and $\dim(\Rad(T)) \leq 2$ since $S_2+S_3$ is non-degenerate of dimension $4$.
Thus
\[ \dim(T^\perp) - \dim(\Rad(T^\perp)) = 2n - \dim(T) - \dim(\Rad(T)) \geq 2n-6-2\geq 10-8 \geq 2. \]
This shows that $T^\perp$ is a possibly degenerate symplectic space that contains non-degenerate symplectic subspaces of dimension $2$ or more.
Then we can take $W\in \G(T^\perp)$, and this $W$ triangulates the $5$-gon into two squares and a triangle.
This finishes the proof of the lemma.
\end{proof}

\begin{lemma}
\label{lm:4gonsDim5orMore}
Let $C$ be a $4$-gon in the graph $\G(V)$ of a symplectic space $V$ of dimension $2n$, with $n\geq 5$.
Then $C$ is null-homotopic in $\F(V)$.
\end{lemma}

\begin{proof}
Suppose that $C$ is the sequence of vertices $S_0,S_1,S_2,S_3,S_4=S_0$ with $S_i \perp S_{i+1}$ for all $0\leq i < 4$.
If $\F( (S_0+S_2)^\perp) $ is connected, we can get a path between $S_1,S_3\in\F( (S_0+S_2)^\perp)$, which triangulates the $4$-gon.
This holds, for example, when the sum $S_0+S_2$ is non-degenerate.
We can proceed analogously for $S_1+S_3$.

Therefore, by Theorem \ref{thm:connected} it remains to analyse the situation where both $S_0+S_2$ and $S_1+S_3$ are degenerate and their orthogonal complements have a non-degenerate part of dimension $ < 6$.
Note that if $S_2\in \E_2(S_0)$, then $\dim(\Rad(S_0+S_2)) = 1$, so 
\[ \dim( (S_0+S_2)^\perp) - \dim(\Rad(S_0+S_2)) = 2n-3-1 \geq 6.\]
Thus in this case, $\F( (S_0+S_2)^\perp)$ is connected (see Theorem \ref{thm:connected}).

Hence $S_0+S_2$ and $S_1+S_3$ are both degenerate of dimension $4$ (that is, $S_2\in \E_3(S_0)$ and $S_3\in\E_3(S_1)$), and the non-degenerate part of their orthogonal complements has dimension at most $4$, which means that
\[4\geq \dim( (S_0+S_2)^\perp)- \dim(\Rad(S_0+S_2)) = (2n-4)-2.\]
This shows that $n = 5$.

% We consider $S:=S_0+S_1+S_2+S_3$, and we prove that $\F_1(S^\perp)\neq\emptyset$.
% That is, we show that $\Rad(S) < S^\perp$.
% Assume otherwise: then $\dim(S)+\dim(\Rad(S)) = 10$.

%In order to simplify notation, we take $\Psi$ to be the canonical form and, by transitivity, we assume without loss of generality that $S_0 = E_1 = \gen{e_1,e_{n+1}}$ and $S_1 = E_2 = \gen{e_2,e_{n+2}}$.
Suppose that $e,f$ and $e',f'$ are symplectic bases of $S_0$ and $S_1$ respectively.
Let $W = S_2 = \gen{w,u}$ and $T = S_3 = \gen{x,y}$, where the bases are taken as in Table \ref{sixCases}.
Since $W\perp S_1$, $T\perp S_0$ $W\not\perp S_0$ and $T\not\perp S_1$, we have:
\[\begin{cases}
w = w_ee + w_ff + \hat{w},\\
u = u_ee + u_ff + \hat{u},\\
x = x_{e'}e' + x_{f'}f' + \hat{x},\\
y = y_{e'}e' + y_{f'}f' + \hat{y},
\end{cases}\]
where the hat-vectors lie in $(S_0+S_1)^\perp$.
Indeed, $\gen{\hat{w}, \hat{u}, \hat{x},\hat{y}} $ is a totally isotropic subspace by the orthogonality relation $T\perp W$ and since $T\in\E_3(S_1)$ and $W\in \E_3(S_0)$ (see Table \ref{sixCases}).
We also have that $\hat{w}=w_{S_0},\hat{u}=u_{S_0}$ are linearly independent.
% Then $S = \gen{e_1,e_2,e_{n+1},e_{n+2},\hat{w},\hat{u},\hat{x},\hat{y}}$.
% Moreover, since $W\in \E_3(E_1)$ and $T\in \E_3(E_2)$, by Table \ref{sixCases} we have that the sets $\{ \hat{w}$, $\hat{u} \}$, $\{\hat{x},\hat{y}\}$, $\{w_1+w_{n+1},u_1 + u_{n+1}\}$, $\{x_2 + x_{n+2}, y_2 + y_{n+2}\}$ are linearly independent.

Take $w'=w-\hat{u}$, $u' = u - \hat{u}$.
Note that $w'_{S_0} = \hat{w}-\hat{u}\neq 0$ by linearly independence, ${u}'_{S_0} = 0$ by definition, and $\Psi(w',u') = \Psi(w,u) = 1$.
Thus $W' = \gen{w',u'}\in\E_2(S_0)$.
Also note that $W'\perp T$ and $W'\perp S_1$.
Then the square $C'=S_0,S_1,W',T,S_0$ is null-homotopic as we saw above.

On the other hand, note that $W+W'=\gen{w,u,\hat{u}}$ has dimension $3$, that is $W'\in \E_2(W)$.
This means that also the square $C''=W',S_1,W,T,W'$ is null-homotopic.
Since $C$ is homotopic to the concatenation of the squares $C'$ and $C''$, we conclude that $C$ is null-homotopic.
\end{proof}

We have proved the following theorem:

\begin{theorem}
\label{thm:simplyConnected}
Let $V$ be a symplectic space of dimension $2n\geq 10$.
Then $\F(V)$ is simply connected.
\end{theorem}

\section{Application of Garland's method}

In this section, we apply Garland's method to the frame complex of a symplectic space over a finite field.
%Recall that we have characterised the homotopy type of $\F(V)$ when $n=2,3$ in Theorem \ref{thm:connected}.

Let $\kk$ be any ring.
For $m\geq -1$, a topological space $X$ is $m$-connected over $\kk$ if $\widetilde{H}_k(X,\kk) = 0$ for all $k\leq m$.
A simplicial complex $K$ of dimension $d$ is spherical over $\kk$ if it is $(d-1)$-connected over $\kk$.
In addition, $K$ is Cohen-Macaulay over $\kk$ if for every simplex $\sigma\in K$ (including $\sigma=\emptyset$), the link $\Lk_K(\sigma)$ is spherical over $\kk$ of dimension $d-|\sigma|$.

We recall now Garland's theorem.
See \cite{BS, Garland}.

\begin{theorem}
[Garland]
Let $K$ be a finite simplicial complex of dimension $d$ and let $0\leq i\leq d$.
Assume the following conditions hold for every $(i-1)$-dimensional simplex $\sigma\in K$:
\begin{enumerate}
    \item The link $\Lk_K(\sigma)$ is connected of dimension $\geq 1$.
    \item The smallest non-zero eigenvalue of the normalised Laplacian of the $1$-skeleton of $\Lk_K(\sigma)$ is strictly bigger than $\frac{i}{i+1}$.
\end{enumerate}
Then $\widetilde{H}_i(K,\QQ) = 0$.
\end{theorem}

Recall that the normalised Laplacian of a graph $\G$ is $\Id - D^{-1}A$, where $D$ is the diagonal matrix with vertex degrees in the diagonal, and $A$ is the adjacency matrix of $\G$.

\medskip

We apply Garland's method to the frame complex $\F(V)$.
Observe that if $\sigma\in \F(V)$ is an $i$-frame, then its link is naturally isomorphic to $\F(\gen{\sigma}^\perp)$, where $\gen{\sigma}^\perp$ is a symplectic space of dimension $2(n-i)$.
Hence this link is connected if and only if $n-i\geq 3$ by Theorem \ref{thm:connected}.

On the other hand, we have computed the eigenvalues of $\G(V)$, the $1$-skeleton of $\F(V)$, for every symplectic space $V$ over a finite field.
Thus, by Theorem \ref{thm:eigenvalues}, the smallest non-zero eigenvalue of the normalised Laplacian of $\G(V)$ for $n\geq 3$ is
\[ \lambda_{\min}(n,q) := 1 - \frac{q^{3n-6}}{d_n}.\]

So we fix $0\leq i\leq n-3$ (i.e. such that $n-i\geq 3$), and check the condition of Garland's theorem: since the link of every $(i-1)$-simplex is isomorphic to the frame complex of a symplectic space of dimension $2(n-i)$, the condition on the minimal eigenvalue is
\begin{equation}
    \label{eq:garlandCondition1}
    \lambda_{\min}(n-i,q) > \frac{i}{i+1} = 1 - \frac{1}{i+1}.
\end{equation}
The above inequality is equivalent to:
\begin{equation}
    \label{eq:garlandCondition2}
    \frac{d_{n-i}}{q^{3(n-i)-6}} > i+1.
\end{equation}
Let $j := n-i$.
Thus (\ref{eq:garlandCondition2}) is equivalent to:
\begin{equation}
    \label{eq:garlandCondition3}
    P_j(q) := \frac{(q^{2j-2}-1)}{q^{j-2}(q^2-1)} + j - 1 > n.
\end{equation}

We note that the $P_j(q)$ are monotone in $j$.

\begin{lemma}
Let $3\leq j < i\leq n$ and $q\geq 2$.
Then $P_j(q) + (i-j) < P_i(q)$.

Thus, by Garland's method, if $P_j(q) > n$ then $\F(V)$ is $(n-j)$-connected over $\kk$.
\end{lemma}

\begin{proof}
We compute:
\begin{align*}
    P_i(q) - P_j(q) - (i-j) & = \frac{(q^{2i-2}-1)q^{j-2} - (q^{2j-2}-1)q^{i-2}}{q^{j+i-4}(q^2-1)} =  \frac{ (q^{j+i-2}+1)( q^{i-2}- q^{j-2})}{q^{j+i-4}(q^2-1)}.
\end{align*}
This expression is positive since $q^{i-2}>q^{j-2}$.

The In particular part follows by Eq. (\ref{eq:garlandCondition3}) and the application of Garland's method.
\end{proof}

We compute $P_j(q)$ for some values of $3\leq j\leq n$:
\begin{itemize}
\item If $j = 3$, then $P_3(q) = q+2+q^{-1}$. So $n < P_3(q)$ if and only if $n < q+3$.
\item If $j = 4$, then $P_4(q) = q^2+4+q^{-2}$.
Thus $n < P_4(q)$ if and only if $n < q^2+4$.
\item Write $n = 2n'+\epsilon$, with $\epsilon\in\{0,1\}$.
%Suppose further that $n'> 4$ if $q = 2$, and $n' \geq 4$ if $q > 3$. 
Let $j := n'+\epsilon$, and consider the expression:
\[ P_j(q) - n = \frac{q^{2n'+2(\epsilon-1)}-1}{q^{n'+\epsilon-2}(q^2-1)} - n' - 1.\]
From basic analysis, we see that $P_j(q) > n $ if and only if the bounds displayed in Table \ref{tab:halfnconnected} hold.
Thus, $\F(V)$ is $n-j=[n/2]$-connected over a field of characteristic $0$ if one of the following holds:
\begin{enumerate}
    \item $q = 2$ and $n\geq 7$,
    \item $q = 3$ and $n\geq 5$, $n\neq 6$,
    \item $q \geq 4$ and $n\geq 5$.
\end{enumerate}
\begin{table}[ht]
    \centering
    \begin{tabular}{|c|c|c|}
    \hline
          & $\epsilon = 0$ & $\epsilon = 1$ \\
    \hline
    $q = 2$ & $n'\geq 4$ & $n'\geq 3$\\
    $q = 3$ & $n'\geq 4$ & $n'\geq 2$\\
    $q\geq 4$ & $n'\geq 3$ & $n'\geq 2$\\
    \hline
    \end{tabular}
    \caption{Positivity of $P_{n'+\epsilon}(q)-n$ in terms of $q$, $n'$ and $\epsilon$.}
    \label{tab:halfnconnected}
\end{table}
\end{itemize}

Since an interval of the poset $\redF(V)$ is either a wedge of spheres or isomorphic to $\redF(\gen{\sigma}^\perp)$ for some frame $\sigma$, we can conclude Cohen-Macaulayness over a field of characteristic $0$ if $n<q+3$.
%From these computations, we conclude:
We summarise these conclusions in the following corollary.

\begin{corollary}
\label{coro:garland}
Let $V$ be a symplectic space of dimension $2n$, $n\geq 4$, over a finite field $\GF{q}$, and $\kk$ a field of characteristic $0$.
\begin{enumerate}
    \item If $n < q + 3$ then $\redF(V)$ is Cohen-Macaulay over $\kk$.
    \item If $n < q^2+4$ then $\redF(V)$ is $(n-4)$-connected over $\kk$.
    \item If $n\geq 7$ for $q=2,3$, or $n\geq 5$ for $q\geq 4$, then $\F(V)$ is $[n/2]$-connected over $\kk$.
\end{enumerate}
\end{corollary}

\section{Final remarks}

In general, we can show that the poset $\redF(V)$ is not Cohen-Macaulay over any ring $\kk$.
This is the case if the dimension $2n$ is big enough since there some links will not be spherical.
Concretely, for a given $n_0$, we can prove that $\widetilde{H}_{n_0-3}(\F(V),\QQ)\neq 0$, where $V$ is a symplectic space of dimension $2n_0$ over $\GF{q}$.
Then, for an arbitrary $V$ of dimension $2n\geq 2n_0$, we take a frame $\sigma\in\redF(V)$ such that $|\sigma|=n-n_0$.
The link $\redF(V)_{>\sigma}$ is isomorphic to $\redF(\gen{\sigma}^\perp)$ and $\gen{\sigma}^\perp$ is a symplectic space of dimension $2n_0$, so $\widetilde{H}_{n_0-3}(\redF(V)_{>\sigma},\kk)\neq 0$, that is, $\redF(V)_{>\sigma}$ is not spherical over $\kk$.

\begin{proposition}
Let $V$ be a symplectic space of dimension $2n$ over a finite field $\GF{q}$.
\begin{enumerate}
    \item If $n > q^2(q^2+1) + n(n-2)q^{-4}(q^4+q^2+1)^{-1}$ then $\widetilde{H}_{n-3}(\F(V),\QQ)\neq 0$.
    \item If $n \geq m$, where $m$ satisfies the bound of item (1), $\F(V)$ is not Cohen-Macaulay over any ring.
\end{enumerate}
\end{proposition}

\begin{proof}
Item (1) follows from a basic counting on the dimensions of the chain groups.
Namely, let $f_m = |\F_m(V)|$, which represents the dimension of the $(m-1)$-th chain group of $\F(V)$ over $\QQ$.
These values can be computed from Eq. (\ref{eq:stabilizer}).
Then, if the bound
\begin{equation}
\label{eq:nonzeroHnminus3}
f_n+f_{n-2} > f_{n-1}+f_{n-3}   
\end{equation}
holds, $\widetilde{H}_{n-3}(\F(V),\QQ)\neq 0$.
Now it is straightforward to verify that Eq. (\ref{eq:nonzeroHnminus3}) holds if and only if the bound of item (1) holds.

Item (2) follows from the discussion previous to the statement of this proposition.
\end{proof}

\begin{remark}
For a symplectic space $V$ of dimension $8$, that is $n=4$, we have not described the fundamental group of $\F(V)$.
Note that if the underlying field is finite of size $q$, then $P_3(q)>4$, so by Garland's theorem $\F(V)$ is $1$-connected over a field of characteristic $0$.
That is, $\widetilde{H}_1(\F(V),\ZZ)$ is a finite abelian group.
This might suggest that $\F(V)$ is simply connected.
Moreover, by the results of \cite{Zuk}, $\pi_1(\F(V))$ has Kazhdan's property (T).
\end{remark}

% \begin{remark}
% Recall that the poset of proper non-degenerate subspaces of a symplectic space $V$ of dimension $2n$ over a finite field of size $q>2$ is homotopy Cohen-Macaulay by \cite{Das}.
% By \cite[Prop. 2.13]{PW22} and Corollary \ref{coro:garland}, for $q = 2$, we can conclude that the poset of proper non-degenerate subspaces of $V$ is Cohen-Macaulay over a field of characteristic $0$ if in addition $n<5$.
% It would be interesting to see if the Cohen-Macaulay property on this poset can be extended to every $n$ if $q = 2$.
% \end{remark}

Finally, we wonder if the results on Cohen-Macaulayness can be extended to the homotopy context.
That is, if $n < q+3$ implies that $\redF(V)$ is homotopy Cohen-Macaulay.

\end{document}